\documentclass[12pt]{amsart}

\usepackage[backend=biber,style=numeric,sorting=nyt,doi=false,url=false,isbn=false,giveninits=true]{biblatex}
\usepackage{tabularx,booktabs}
\usepackage{caption}
\usepackage{amsmath}
\usepackage{bbm}
\usepackage{amsfonts}
\usepackage{amscd}
\usepackage{amsthm}
\usepackage{amssymb} 
\usepackage{cancel}
\usepackage{geometry}
\usepackage{latexsym}
\usepackage{eufrak}
\usepackage{euscript}
\usepackage{epsfig}
\usepackage{graphics}
\usepackage{array}
\usepackage{enumerate}
\usepackage{dsfont}
\usepackage{color}
\usepackage{wasysym}
\usepackage{hyperref}
\usepackage{multirow}
\usepackage{pdfsync}
\newcommand{\bel}[1]{\begin{equation}\label{#1}}

\newcommand{\be}{\begin{equation}}

\newcommand{\qe}{\end{equation}}
\newcommand{\R}{{\mathbb R}}
\newcommand{\N}{{\mathbb N}}
\newcommand{\Z}{{\mathbb Z}}
\newcommand{\C}{{\mathbb C}}

\newcommand{\Hmm}[1]{\leavevmode{\marginpar{\tiny%
$\hbox to 0mm{\hspace*{-0.5mm}$\leftarrow$\hss}%
\vcenter{\vrule depth 0.1mm height 0.1mm width \the\marginparwidth}%
\hbox to
0mm{\hss$\rightarrow$\hspace*{-0.5mm}}$\\\relax\raggedright #1}}}

\newtheorem{theorem}{Theorem}[section]

\newtheorem{lemma}[theorem]{Lemma}
\newtheorem{corollary}[theorem]{Corollary}
\newtheorem{definition}[theorem]{Definition}
\newtheorem{remark}[theorem]{Remark}

\newtheorem{prop}[theorem]{Proposition}

\begin{document}

\title[The wave equation on lattices and Oscillatory Integrals
]{The wave equation on lattices and\\ Oscillatory Integrals
}

\author{Cheng Bi}
\address{Cheng Bi: School of Mathematical Sciences, Fudan University, Shanghai 200433, China}
\email{\href{mailto:cbi21@m.fudan.edu.cn}{cbi21@m.fudan.edu.cn}}

\author{Jiawei Cheng}
\address{Jiawei Cheng: School of Mathematical Sciences, Fudan University, Shanghai 200433, China}
\email{\href{mailto:chengjw21@m.fudan.edu.cn}{chengjw21@m.fudan.edu.cn}}

\author{Bobo Hua}
\address{Bobo Hua: School of Mathematical Sciences, LMNS, Fudan University, Shanghai 200433, China; Shanghai Center for Mathematical Sciences, Fudan University, Shanghai 200433, China}
\email{\href{mailto:bobohua@fudan.edu.cn}{bobohua@fudan.edu.cn}}

\begin{abstract}
In this paper, we establish sharp dispersive estimates for the linear wave equation on the lattice $\Z^d$ with dimension $d=4$. Combining the singularity theory with results in uniform estimates of oscillatory integrals, we prove that the optimal time decay rate of the fundamental solution is of order $|t|^{-\frac{3}{2}}\log |t|$, which is the first extension of P. Schultz's results \cite{S98} in $d=2,3$ to the higher dimension. Moreover, we notice that the Newton polyhedron can be used not only to interpret the decay rates for $d=2,3,4$, but also to study the most degenerate case for all odd $d\geq 3$.
Furthermore, we prove $l^p\rightarrow l^q$ estimates as well as Strichartz estimates and give applications to nonlinear wave equations.
\end{abstract}

\maketitle

\section{Introduction}\label{sec-intro}

Discrete analogs of partial differential equations on graphs have been extensively studied in recent years. Most works focus on elliptic and parabolic equations, see  e.g. \cite{B17,Gri18}. References related to wave equations include \cite{FT04,LX19,LX22}.
In this article, we
consider the following nonlinear wave equation on the $d$-dimensional lattice $\mathbb{Z}^d$,
\begin{equation}\label{equ-orig}
    \left\{
    \begin{aligned}
        & \partial_{t}^2 u(x,t) - \Delta u(x,t) = F(u(x,t)), \\
        & u(x,0) = g(x),\quad\partial_t u(x,0) = f(x),
    \end{aligned}
    \right.
\end{equation}
where $x=(x_1,\cdots,x_d)\in \mathbb{Z}^d$ and $t\in \R$. The discrete Laplacian $\Delta$ is defined by
\begin{equation*}
    \Delta u (x,t) := \sum_{y\in\Z^d ,\, \mathrm{d}(x,y) = 1} \Big( u(y,t) - u(x,t)\Big),\quad \mbox{with}\ \  \mathrm{d}(x,y) = \sum_{j=1}^d |x_j-y_j|.
\end{equation*}

There are a number of physical applications where this equation appears in a natural way, mainly to describe the behaviour of wave propagation. One of the best known of them is the description of the vibrations of atoms inside crystals. A fundamental model is the monotonic chains, in which each atom vibrates as a simple harmonic oscillator and only feels the force of its nearest neighbours, see \cite{D93,FH10,Z72}.

As is shown in \cite{HH20}, there exists a nontrivial Tychnoff-type solution to \eqref{equ-orig} with $F,g,f \equiv 0$. In particular, this solution does not belong to $l^2$, hence it is not a ``physical'' solution. Therefore, in what follows we restrict attention to the so-called semigroup solution, see \eqref{equ-solu-semigroup} below.

Obtaining dispersive inequalities for linear dispersive equations usually serves as the first step in the study of nonlinear problems. Generally, it amounts to establishing a decay estimate for the $l^{\infty}$ norm of the solution in terms of time and the $l^1$ norm of the initial data.  
On Euclidean space $\R^d$, it is well-known that the decay rate of the solution to wave equation is of order $|t|^{-\frac{d-1}{2}}$, while for Schr\"{o}dinger equation, it decays like $|t|^{-\frac{d}{2}}$, see e.g. \cite{GV95,KT98,S77}. These results have been extended to more general framework such as Heisenberg groups and H-type groups, see for instance H. Bahouri et al. \cite{BFG16, BGX00} and M. Del Hierro \cite{D05}.

From now on we focus on \eqref{equ-orig}. Let $F,g \equiv 0$ and $\mathbb{T}^d := [-\pi,\pi]^d$ be the torus, we get the following Green's function by the discrete Fourier transform (see Section \ref{sec-preli}). 
\begin{equation}\label{equ-gree}
    G(x,t) := \frac{1}{(2\pi)^d}\int_{\mathbb{T}^d}  e^{ix \cdot \xi} \; \frac{\sin (t \, \omega (\xi))}{\omega (\xi)}\,d\xi,\quad \mbox{with} \ \ \omega(\xi) = \left({\sum_{j=1}^{d} (2-2\cos \xi_j)}\right)^{\frac{1}{2}},
\end{equation}
where $(x,t)\in\Z^d\times\R$ and $x\cdot \xi=\sum_{j=1}^d x_j\xi_j$ is the usual inner product.

To establish the dispersive estimate, we look for possible value of $\beta<0$  such that
\begin{equation}\label{Green}
    |G(x, t)|\leq C(1+|t|)^{\beta},\quad\forall\, (x,t)\in\Z^d\times\R,
\end{equation}
where the constant $C$ only depends on $d$ (or $C=C(d)$, in short). The symbols $C, c$ will be used throughout to denote positive constants, which may vary from one line to the next. Analogous results as (\ref{Green}) have been studied for other dispersive equations with Cauchy data such as the discrete Schr\"{o}dinger equation (DS, in short) 
\begin{equation*}
    i \partial_t u(x,t) + \Delta u(x,t) = 0
\end{equation*}
and the discrete Klein-Gordon equation (DKG)
\begin{equation}\label{equ-mdkg}
    \partial_t^2 u(x,t) - \Delta u(x,t) +m_*^2 \, u(x,t)= 0 ,
\end{equation}
where $m_*>0$ is the mass parameter. These equations are closely related to our model, that is, the discrete wave equation (DW), which is the vanishing mass limit ($m_*\rightarrow 0$) of the DKG.

For the DS, due to the special form of its Green's function, one can separate variables to reduce the problem to the case $d=1$, and obtain the sharp decay rate of $|t|^{-\frac{d}{3}}$ on $\mathbb{Z}^d$, see \cite[Theorem 3]{SK05}.
Unfortunately, the DKG and the DW fail to have the separation-of-variables property, which leads to more complicated analysis as dimension increases. For the DKG, A. Stefanov and P. G. Keverekidis \cite[Theorem 5]{SK05} established a sharp decay estimate of $|t|^{-\frac{1}{3}}$ on $\mathbb{Z}$. Later, V. Borovyk and M. Goldberg \cite[Corollary 2.5]{BG17} computed a decay of $|t|^{-\frac{3}{4}}$ on $\Z^2$.
Recently, J.-C. Cuenin and I. A. Ikromov \cite[Theorem 1]{CI21} extended the results to dimensions 2, 3 and 4. 

However, we shall see that the situation in the DW is harder than that of the DKG. 
Progress in the DW mainly comes from P. Schultz, who settled the cases $d=2$ and 3 in \cite{S98} with decay rates $|t|^{-3/4}$ and $|t|^{-7/6}$, respectively. He analyzed (\ref{equ-gree}) as well as a related oscillatory integral,
\begin{equation}\label{equ-inte}
    I(v,t) := \frac{1}{(2\pi)^d}\int_{\mathbb{T}^d} e^{i t \phi(v,\xi)}\frac{1}{\omega(\xi)}\, d\xi,\quad \mbox{with}\quad \phi(v,\xi):=v \cdot \xi - \omega(\xi),
\end{equation}
where $(v,t)\in\R^d\times\R$. Notice that the phase and amplitude are not smooth at the origin, which brings some difficulties. In light of (\ref{Green}) and the relation $G(x,t)=-\mbox{Im}\,I(x/t,t)$ (see (\ref{equ-Im}) below), in many occasions it is more convenient to establish inequalities of the following type,
\begin{equation}\label{equ-I}
    |I(v,t)| \leqslant C(1+|t|)^{\beta},\quad\forall\, (v,t)\in\R^d\times\R.
\end{equation}

We first look for critical points of $\phi(v,\cdot)$, which suggests a partition in the velocity space of $v$ into the following regions,
\begin{gather}\label{regions}
    \mbox{the exterior} \; (|v|>1), \;\;\mbox{the vicinity}\; (|v|\approx 1) \;\;\mbox{and the interior} \;(|v|<1)
\end{gather}
of the light cone, see also Section \ref{ssec-analysis}. The first two regions are handled by \cite{S98} in all dimensions $d\geq 2$. However, the situation becomes complicated inside the light cone, in which the degenerate critical points appear. By this we mean Hess$_\xi\phi(v_0,\xi_0)$, i.e. the Hessian of $\phi(v_0,\cdot)$ at $\xi=\xi_0$, is singular, where $v_0$ is some velocity in the light cone and the wave number $\xi_0$ is a critical point of $\phi(v_0,\cdot)$.
In this case, the stationary phase method breaks down and we have to consider each dimension one-by-one.  
 
In this article, we mainly consider the degenerate cases inside the light cone. This problem can be viewed as the stability of oscillatory integral under phase (linear) permutations. More precisely, for any $v_0\in\R^d$, since 
\begin{equation}\label{equ-lin-per}
    \phi(v,\xi)=(v-v_0)\cdot \xi+\phi(v_0,\xi),
\end{equation}
we first get the decay rate of \eqref{equ-inte} with $v=v_0$, then we prove that the estimate holds uniformly for $v$ in some neighborhood of $v_0$. Finally, we get the uniform estimate in $v\in\R^d$ by a finite covering on the velocity space. 
 
We notice that decay rates in $d=2,3,4$ can be well interpreted by the Newton polyhedron of $\phi$, see Section \ref{ssec-alt234}. Moreover, when 
$d=3,4$ the decay is governed by a single velocity $$v_0=\left(\tfrac{1}{\sqrt{2d}},\,\cdots,\,\tfrac{1}{\sqrt{2d}}\right)\in\R^d,$$ at which the most degenerate case appears, that is, the rank of Hess$_\xi \phi(v_0,\xi_0)$ attains the minimum $1$, see Lemma \ref{lem-dege}, \eqref{equ-most-dege} and \hyperref[table-2]{Table 1}. This fact may be useful in predicting decay rates in higher dimensions, and we obtain the decay rates for all odd $d\geq 3$ in such cases, see Theorem \ref{thm-conj} below.

The uniform estimate of oscillatory integral is noteworthy in itself. Roughly speaking, we consider the following integral
\begin{equation}\label{equ-OI}
    J(t,S+P,\, \psi):=\int_{\R^d}e^{it(S(x)+P(x))}\psi(x)\,dx,
\end{equation}
with proper phase $S$, amplitude $\psi$ and perturbation $P$. A natural question is whether the decay of $J(t,S,\,\psi)$ as $t\rightarrow\infty$ could extend to $J(t,S+P,\,\psi)$ when $P$ is ``small" enough. Unidimensional uniform estimate goes back to I. M. Vinogradov \cite{V80} and J. G. Van der Corput \cite{VDC23}. When $d=2$, the adapted coordinate system is important. A. N. Varchenko \cite{V76} proved the existence of such coordinate system for analytic functions (without multiple
components) in his pioneering work studying the connection between oscillatory integral and Newton polyhedra. Based on this, V. N. Karpushkin \cite{K84} gave an affirmative answer to the question above in the real-analytic setting. Later, Ikromov and D. M\"{u}ller \cite{IM11} extended the results to smooth phase of finite type when $P$ is linear. Besides, it is also effective to analyze the singularities, see for instance Arnold et al. \cite{AGV12} and J. J. Duistermaat \cite{D74}.

In higher dimensions, however, examples in \cite{V76} give an negative answer to the question above and show that the adapted coordinates may not exist. Moreover, even for $P=0$, it is still extremely difficult to determine the precise asymptotic behaviour of (\ref{equ-OI}) in many cases. Many results in such cases are motivated by Newton polyhedra, see e.g. \cite{V76}, D. H. Phong and E. M. Stein \cite{PS97} and M.Greenblatt \cite{G14}. For more details about the Newton polyhedra, we refer to Section \ref{sec-new}.

Since the seminal paper of Schultz \cite{S98} on dispersive estimates for the DW on $\Z^d$ with $d=2,3$, whether these results can be extended to higher dimensional lattices has been left as a major open problem. In this article, we extend the results of \cite{S98} to $d=4$, serving as the first step towards the above open problem. Besides, the most degenerate case in high dimensions and the Strichartz estimates are also studied for the first time. Our main result is the following.

\begin{theorem}\label{main-thm}
    There exists $C>0$ such that 
    \begin{equation*}
        |G(x,t)| \leqslant C(1+|t|)^{-\frac{3}{2}} \log (2+|t|),\quad \forall\, (x,t)\in\Z^4\times\R.
    \end{equation*}
\end{theorem}

\begin{remark}\label{remark-main} 
    The estimate is sharp in the sense that there exist some vector $v_0 \in \R^4$ and a constant $C>0$ such that (cf. \cite{K94})
    \begin{equation*}
      |I(v_0,t)| \geqslant C\, t^{-\frac{3}{2}}\log t,\quad\mbox{as}\quad t \rightarrow +\infty.
    \end{equation*} 
\end{remark}

Our proof strategy is as follows. First, for all $d\ge 2$ we give the characterization of the degenerate critical points of $\phi$. Then we give a decomposition of the phase in the most degenerate case, which is also useful in establishing the dispersive inequality for all odd $d\ge 3$. Finally, when $d=4$, using proper coordinate changes, we reduce the phase to simple polynomials by considering several cases separately. Then we finish by combining the singularity theory with some results on stability of the oscillatory integral, which firstly appeared in \cite{K83}. 

In the terminology of Arnold \cite{AGV12}, the stable singularities which will appear in the proof are $A_k$ ($k\geq 1$) and $D_4^-$. Besides, the  polynomial $\xi_1\xi_2\xi_3$ plays an important role in the most degenerate case. Moreover, we point out that the pattern of our proof may be useful in the study of dispersive equations on general graphs.

Theorem \ref{main-thm} directly leads to the $l^p\rightarrow l^q$ estimates.
In the sequel, for any $a>0$, the notation ``$a^-$" means that one can choose any $\epsilon>0$ to replace $a$ with $a-\epsilon$. 

\begin{theorem}\label{thm-lplq}
   Let $d=4$ and $u$ be the solution to (\ref{equ-orig}) with $F,g \equiv 0$. If $1 \leqslant p < q \leqslant +\infty$ with $1/p-1/q \geqslant 1/2$, then there exists $C=C(p,q)$ such that
    \begin{equation*}
    \|u(t,\cdot)\|_{l^q} \leqslant C(1+|t|)^{-\beta_{p,q}}\|f\|_{l^p},\ \  \forall\,t\in\R,\quad\mbox{with}\ \ \beta_{p,q}=3^-\left(\tfrac{1}{p}-\tfrac{1}{q}-\tfrac{1}{2}\right).
    \end{equation*}
      Furthermore, if $(1/p,1/q)$ lies on the segment with vertices $(3/4,1/2)$ and $(1,0)$, then there exists $C=C(p,q)$ such that\begin{equation*}
    \|u(t,\cdot)\|_{l^q} \leqslant C(1+|t|)^{-\zeta_q}\|f\|_{l^p},\ \  \forall\,t\in\R,\quad\mbox{with} \ \ \zeta_q=\left(\tfrac{3}{2}\right)^- \left(1-\tfrac{2}{q}\right).
    \end{equation*}
\end{theorem}
This result can be used to prove the global existence of the solution for nonlinear equations with power type nonlinearity, see Section \ref{ssec-non}.

Remark that the DKG is considered in \cite{CI21} for $d=2,3,4$ with
a proof relying also on the analysis of singularities. The analog of \eqref{equ-inte} is 
\begin{equation}\label{equ-DKG}
    \widetilde{I}(v,t)=\int_{\mathbb{T}^d}e^{it(v\cdot\xi-\widetilde{\omega}(\xi))}\frac{1}{\widetilde{\omega}(\xi)}\,d\xi,\quad \mbox{with} \quad \widetilde{\omega}(\xi) = \left({m_*^2+\sum_{j=1}^d (2-2\cos \xi_j)}\right)^{\frac{1}{2}},
\end{equation}
where $m_*$ is as in \eqref{equ-mdkg}.  The situation of the DW is more complicated than that of the DKG, which is essentially because that the phase and amplitude in (\ref{equ-DKG}) are more regular compared with \eqref{equ-inte}. 
For the same reason, in the following Strichartz estimates we need to overcome more difficulties caused by the operator $\frac{1}{\sqrt{-\Delta}}$, while the corresponding operator for the DKG, i.e. $\frac{1}{\sqrt{\boldsymbol{1}-\Delta}}$, is easier to deal with, see \cite{CI21} and Section \ref{ssec-stri/tao}. In what follows, for any $1\le p\le +\infty$, let $p'$ be the conjugate index of $p$.

\begin{theorem}\label{thm-stri}
    Let $d=4$ and $u$ be the solution to (\ref{equ-orig}). If indices $q,r,\widetilde{q},\widetilde{r}$ satisfy
    \begin{equation}\label{tao}
        q,r,\tilde{q},\tilde{r} \geqslant 2,\quad \frac{1}{q} < \frac{3}{2}\left(\frac{1}{2}-\frac{1}{r}\right)\quad\mbox{and}\quad
        \frac{1}{\tilde{q}} < \frac{3}{2}\left(\frac{1}{2}-\frac{1}{\tilde{r}}\right),
    \end{equation}
    then there exists $C=C(q,r,\widetilde{q},\widetilde{r})$ such that
    \begin{equation*}
    \|u\|_{L^q_t l^r} \leqslant C\left(\|g\|_{l^2} + \|f\|_{l^{\frac{4}{3}}} + \|F\|_{L_t^{\tilde{q}'} l^{\frac{4\tilde{r}'}{4+\tilde{r}'}}}\right)
    \end{equation*}
\end{theorem}

This result is proved by a combination of Theorem \ref{main-thm}, the $l^p$ boundedness of $\frac{1}{\sqrt{-\Delta}}$ and the result in M. Keel and T. Tao \cite{KT98}.
The analog in the case $d=3$ can be proved similarly, see Theorem \ref{thm-stri-d=3}. Moreover, there is another Strichartz estimate following directly from Theorem \ref{thm-lplq}, see Theorem \ref{thm-stri/lplq} and Remark \ref{rmk-compare}. \\

Most of our results focus on $\Z^4$. However, in higher dimensions, we have the following estimates by analyzing a special class of Newton polyhedra.

\begin{theorem}\label{thm-conj}
    For any odd $d\geq 3$, there exists $C=C(d)>0$ such that
    \begin{equation}\label{equ-2d+1}
        |I(v_0,t)| \leqslant C (1+|t|)^{-\frac{2d+1}{6}},\quad \forall\, t\in \R, 
    \end{equation}
    where $v_0=(\tfrac{1}{\sqrt{2d}},\cdots,\tfrac{1}{\sqrt{2d}})\in\R^d.$
\end{theorem}

 As we mentioned before, when $d=3,4$, the decay rates are determined by the most degenerate case. We believe that this claim also holds for $d\ge 5$ and \eqref{equ-2d+1} holds uniformly in $v_0\in \R^d$, which is better than \cite[Remark 1(b)]{CI21}. They conjectured an estimate of order $|t|^{-\frac{2d+1}{6}}\log^{d-4}|t|$ for $d\ge 5$.

This paper is organized as follows. In Section \ref{ssec-basics} we recall basic concepts about the DW, then we give complete proofs for results on uniform estimates of the ocsillatory integrals in Section \ref{ssec-uni}. Section \ref{sec-pf} is devoted to proving Theorem \ref{main-thm}. In Section \ref{ssec-new} we recall some facts of Newton polyhedra, and we give application to the DW in Section \ref{ssec-alt234}, Theorem \ref{thm-conj} is proved in Section \ref{ssec-odd}. In Section \ref{sec-nonli}, we prove Theorem \ref{thm-lplq} and Theorem \ref{thm-stri}, then we give applications to the nonlinear equations. \\

\section{Preliminaries}\label{sec-preli}
\subsection{Basics on $l^p(\Z^d)$ and the DW}\label{ssec-basics}
Let $\mathbb{Z}^d$ be the standard integer lattice graph in $\mathbb{R}^d$. For $p \in [1,\infty]$, $l^p(\Z^d)$ is the $l^p$-space of functions on $\Z^d$ with respect to the counting measure, which is a Banach space with the norm 
\begin{equation*}
    ||h||_{l^p} :=\left\{
    \begin{aligned}
        &\left(\sum_{x \in \Z^d} |h(x)|^p\right)^{\frac{1}{p}},\ p\in[1,\infty),\\
        &\sup_{x\in\Z^d} |h(x)|,\ p=\infty.
    \end{aligned}\right.
\end{equation*}
We shall also use $|h|_p$ to denote the $l^p$ norm of $h$ for notational convenience. For $1\leqslant q,r < \infty$, the mixed space-time Lebesgue spaces $L^q_t l^r$ are Banach spaces endowed with the norms
\begin{equation*}
    ||F||_{L^q_t l^r} := \left( \int_{\R}\left(\sum_{x\in \Z^d}|F(x,t)|^r\right)^{\frac{q}{r}}dt\right)^{\frac{1}{q}},
\end{equation*}
 with natural modifications for the case $q=\infty$ or $r=\infty$. Moreover, for proper functions $h_1,h_2$ on $\mathbb{Z}^d$ we define the convolution product as
\begin{equation*}
    h_1*h_2(x) := \sum_{y \in \mathbb{Z}^d} h_1(x-y)h_2(y),\quad\forall\, x\in\Z^d.
\end{equation*}

The $l^p$ spaces are analogous to the $L^p$ spaces of functions defined on $\R^d$. Many results of the $L^p$ spaces extend to the lattice such as the H\"{o}lder inequality, Young's inequality for convolution and Riesz-Thorin interpolation theorem. One major difference is that the $l^p$ spaces are nested: 
$ l^p \subset l^q$,  $\forall\,1 \leqslant p \leqslant q \leqslant \infty$.

The discrete Fourier transform of a proper function $h$ is given by 
\begin{equation*}
    \mathcal{F}(h)(\xi)=\Hat{h}(\xi) := \sum_{x\in \mathbb{Z}^d} e^{-i\xi \cdot x}h(x),\quad\forall \,\xi \in \mathbb{T}^d,
\end{equation*}
while the inverse transform is 
\begin{equation*}
    \mathcal{F}^{-1}(h)(x)=\Check{h}(x) := \frac{1}{(2\pi)^d} \int_{\mathbb{T}^d} e^{i \xi \cdot x} h(\xi) d\xi,\quad\forall\,x\in \mathbb{Z}^d.
\end{equation*}
See e.g. \cite{W11} for more facts about discrete Fourier analysis. We may use the same notation to denote the Fourier transform and the inverse transform of a distrbution on $\R^d$. Applying the Fourier transform to both sides of (\ref{equ-orig}), we get 
\begin{equation*}
\left\{
    \begin{aligned}
        & \partial_t^2 \Hat{u}(\xi,t)   + \omega(\xi)^2\, \Hat{u}(\xi,t)=0,\\
        & \Hat{u}(\xi,0) = \hat{g}(\xi),~~\partial _t \Hat{u}(\xi,0) = \Hat{f}(\xi),\quad\forall\,\xi\in\mathbb{T}^d.
    \end{aligned}
\right.
\end{equation*}
The solution to this ordinary differential equation is 
\begin{equation*}
    \Hat{u}(\xi,t) = \cos(t\omega)\hat{g}(\xi)+\frac{\sin(t\omega)}{\omega} \hat{f}(\xi),\quad\forall\,\xi\in\mathbb{T}^d,
\end{equation*}
which gives that
\begin{equation*}
    u(x,t)=\frac{1}{(2\pi)^d} \int_{\mathbb{T}^d} e^{i \xi \cdot x} \left(\cos(t\omega)\hat{g}(\xi)+\frac{\sin(t\omega)}{\omega} \hat{f}(\xi)\right) d\xi,\quad\forall\,(x,t)\in\Z^d\times\R.
\end{equation*}

In the notion of operator theory, for any $f,g\in l^2$ and $t\in\R$,
\begin{equation}\label{equ-solu-semigroup}
    u(t)=\cos(t\sqrt{-\Delta})g+\frac{\sin (t\sqrt{-\Delta})}{\sqrt{-\Delta}}f.
\end{equation}

From now on we only consider the solution for zero initial position and a given initial velocity, i.e. $g \equiv 0$, unless otherwise stated. The other case can be treated similarly, see Section \ref{ssec-stri/tao}. Then we get $u=f*G$ with the Green's function $G(x,t)$ defined in (\ref{equ-gree}). Moreover, for $x=vt$ we have
\begin{equation}\label{equ-Im}
     G(x,t)=-\mbox{Im}\, I(v,t)
\end{equation} 
 by the fact that $\omega(\xi) = \omega(-\xi)$ and $I(v,t)=I(-v,t)$ (cf. (\ref{equ-inte})).\\

\subsection{Results on uniform estimates.}\label{ssec-uni}
 We recall some notions and results which were initiated from \cite{K83}. 
 In the sequel, let $B_{\R^d}(\xi,r)$ (resp. $B_{\C^d}(\xi,r)$) be the usual open ball in $\R^d$ (resp. $\C^d$) with center $\xi$ and radius $r$, while $\overline{B}_{\R^d}(\xi,r)$ (resp. $\overline{B}_{\C^d}(\xi,r)$) denotes its closure.

 \begin{definition}
     For any $r,s>0$, the space $\mathcal{H}_r(s)$ is defined as
\[
\mathcal{H}_r(s):=\left\{ P: \
\begin{aligned}
    &P \ \mbox{is holomorphic on}\ B_{\C^d}(0,r) \ \mbox{and continuous}\\ &\mbox{on}\ \overline{B}_{\C^d}(0,r),\ \mbox{and}\ |P({w})|<s,\ \forall \, {w}\in \overline{B}_{\C^d}(0,r)
\end{aligned}
\ \right\}.
\]
 \end{definition}

\begin{definition}\label{def-ue}
    Let $h:\mathbb{R}^d \rightarrow \mathbb{R}$ be real-analytic at 0. We write 
    \begin{equation*}
        M(h) \curlyeqprec (\beta,p) \ \  \mbox{for some} \ (\beta,p)\in (-\infty,0]\times\N,
    \end{equation*} 
    if for any $r>0$ sufficiently small, there exist $\epsilon>0$, $C>0$ and a neighbourhood $A\subset B_{\R^d}(0,r)$ of the origin such that
    \begin{equation}\label{equ-uniform}
         |J(t,h+P,\psi)| \leqslant C (1+|t|)^{\beta} \log^p(|t|+2) \|\psi\|_{C^N(A)},\ \forall\, (t,\psi,P)\in \R\times C_0^{\infty}(A)\times \mathcal{H}_r(\epsilon),
    \end{equation}
    where $J$ is as in \eqref{equ-OI}, $N=N(h)\in \N$ and 
    $$\|\psi\|_{C^N(A)}=\sup\big\{|\partial^{\gamma}\psi(\xi)|:\xi\in A,\,\gamma\in\N^d,|\gamma|\leq N\big\}.$$
\end{definition}

Some conventions are in order. Let $h,h_1,h_2$ be proper functions, $\xi\in\R^d$ and $(\beta_j,p_j)\in (-\infty,0]\times \N$ for $j=1,2$. Then 
\begin{itemize}
    \item[\textbf{(1)}] we write $M(h,\xi) \curlyeqprec (\beta_1,p_1)$, if
    \[
    M(\boldsymbol{\tau}_\xi h) \curlyeqprec (\beta_1,p_1),\quad \mbox{where}\ \ \boldsymbol{\tau}_\xi h(y)=h(y+\xi),\ \ \forall\, y\in\R^d;
    \]
    \item[\textbf{(2)}] we write $M(h_2)\curlyeqprec M(h_1)+(\beta_1, p_1)$, if
    \[
    M(h_1)\curlyeqprec (\beta_1,p_1) \quad\mbox{implies that}\ \   M(h_2)\curlyeqprec (\beta_1+\beta_2,p_1+p_2);
    \]
    \item[\textbf{(3)}]  we write $M(h_2)\curlyeqprec M(h_1)$, if $M(h_2)\curlyeqprec M(h_1)+(0,0)$;
    \item[\textbf{(4)}] we write $M(h)+(\beta_1,p_1)\curlyeqprec(\beta_1+\beta_2,p_1+p_2)$, if $M(h)\curlyeqprec (\beta_2,p_2)$.\\
\end{itemize}

Let $\alpha=(\alpha_1,\cdots,\alpha_d)\in\R^d$ be a weight with $\alpha_j> 0$ for all $j$. 
For any $c>0$, the associated one-parameter family of dilations is defined as
$$
\delta_{c}^\alpha(\xi):=(c^{\alpha_1}\xi_1,\cdots,c^{\alpha_d}\xi_d),\quad\forall \,\xi\in\R^d.
$$
\begin{definition}
     A polynomial $h$ on $\R^d$ is called $\alpha$-homogeneous of degree $\varrho\geq 0$, if 
       $$h\circ\delta_{c}^\alpha(\xi) = c^\varrho h(\xi), \quad\forall \,(\xi,c)\in\R^d\times(0,+\infty).$$
\end{definition}
  Let $\mathcal{E}_{\alpha,d}$ be the set of  $\alpha$-homogeneous polynomials of degree $1$, and $H_{\alpha,d}$ be the set of functions real-analytic at 0 with the associated Taylor's series having the form $\sum_{\gamma\cdot\alpha>1} a_\gamma \xi^\gamma$, i.e. each monomial is $\alpha$-homogeneous of degree greater than 1. 
 
 The following useful lemmas first appeared in \cite{K83}, we include complete proofs here.

\begin{lemma}\label{lem-nocritical}
    Let $h:\R^d\rightarrow \R$ be real analytic at $0$ and $\nabla h(0)\neq 0$, then $$M(h)\curlyeqprec (-n,0),\quad\forall\,n\in\N.$$ 
\end{lemma}

Lemma \ref{lem-nocritical} can be proved directly through integrating by parts.

\begin{lemma}\label{lem-prin}
    Let $h\in \mathcal{E}_{\alpha,d}$ and $P\in H_{\alpha,d}$, then
    $$M(h+P) \curlyeqprec M(h).$$
\end{lemma}

\begin{proof}
    The idea is from \cite[Lemma 1]{K83}.  Firstly, for any $c>0$, we define 
    $$
\mathbf{P}(c,\alpha):=\left\{{w}=(w_1,\cdots,w_d)\in\C^d:|w_j|<\frac{c^{\alpha_j}}{2}\right\}.
$$

Assume that $M(h)\curlyeqprec (\beta,p)$, then for any $r>0$ there exist
$\epsilon_r>0$, $C_r>0$ and $A_r$ as in Definition \ref{def-ue} such that $\mathbf{P}(c_0,\alpha)\cap \R^d\subset A_r$ for some $c_0=c_0({r})>0$.

By the Cauchy inequality (cf. {\cite[Theorem 2.2.7]{Hor90}}), there exists $P_0>0$ such that
$$
P(z)=\sum_{\alpha\cdot \beta>1}a_\beta z^{\beta}\ \ \mbox{on} \ \ B_{\C^d}(0,r), \ \ \mbox{with} \ \ |a_\beta|\leq \frac{P_0}{r^{|\beta|}} \ \ \mbox{for all}\ \beta.
$$

Since $\alpha_j>0$ for all $j$, there exists $\sigma>0$ such that $\alpha\cdot\beta-1\geq \sigma|\beta|$ for all $\beta$ satisfying $\alpha\cdot \beta>1$. Thus we can find $s=s(\epsilon_r)>0$ such that $\mathbf{P}(s,\alpha)\subset B_{\C^d}(0,r)$ and then
\[
\left|\frac{c_0}{s} \, P\circ \delta_{sc_0^{-1}}^\alpha (\xi) \right|
\leq P_0\sum_{\alpha\cdot \beta>1} \left(\frac{s}{c_0}\right)^{\alpha\cdot \beta-1}\leq P_0\sum_{|\beta|\geq 2} \left(\frac{s}{c_0}\right)^{\sigma |\beta|}\leq \frac{\epsilon_r}{2},\ \  \forall\, \xi\in \mathbf{P}(c_0,\alpha).
\]

   By the same token, there exists $\widetilde{\epsilon}>0$ such that 
   $$c_0 s^{-1}\Upsilon \circ \delta_{sc_0^{-1}}^\alpha \in \mathcal{H}_{r}(\epsilon_r/2),\quad\forall\,\Upsilon\in \mathcal{H}_r (\widetilde{\epsilon}).$$
   
   Let $\widetilde{A}=\mathbf{P}(s,\alpha) \cap \R^d$, since $h\in \mathcal{E}_{\alpha,d}$, for any $\psi \in C_0^\infty (\widetilde{A})$ we have
   \[
   \mathbf{I}:=J(t,h+P+\Upsilon,\psi)=\left(\frac{s}{c_0}\right)^{|\alpha|}
   J\left(tsc_0^{-1},h+c_0s^{-1}\left(P_\delta+\Upsilon_\delta\right),\psi_\delta\right)
    \]
    by a change of coordinates, where  we use notation $\varphi_\delta = \varphi \circ \delta_{sc_0^{-1}}^\alpha$ with $\varphi = P,$ $\Upsilon$ or $\psi$. Since  supp$ \, \psi_\delta \subset A_r$ and $M(h)\curlyeqprec (\beta,p)$, we get
    \begin{align*}
     |\, \mathbf{I} \,|&\leq C_r\left(\frac{s}{c_0}\right)^{|\alpha|}\left(1+|t s c_0^{-1}|\right)^{\beta}\log^p (2+|t sc_0^{-1}|)\,\|\psi_{1}\|_{C^{N}}\leq \widetilde{C}(1+|t|)^{\beta}\log^p (2+|t|)\|\psi\|_{C^{N}},
 \end{align*}
 where $\widetilde{C}=C_r\left(sc_0^{-1}\right)^{|\alpha|+\beta}$. Then the proof is completed.\\
 \end{proof}

In the spirit of the stationary phase method, we have:

\begin{lemma}\label{lem-quad}
    Let $m, n\ge 1$ and 
    \begin{equation*}
        h_2(\xi,y) = h_1(\xi) + {Q}(y), \quad \forall\,(\xi,y) \in \R^n\times \R^m,
    \end{equation*}
    where ${Q}(y)= \sum_{j=1}^m c_j y_j^2$ with $c_j=\pm 1$ for all $j$.
    Then
      $$  M(h_2) \curlyeqprec M(h_1)+ \left(-\frac{m}{2},0\right) .$$
\end{lemma}

\begin{proof}
 Assume that $M(h_1)\curlyeqprec (\beta,p)$, for any $r>0$, we can find $\epsilon_r,C_r$ and $A_r$ as in Definition \ref{def-ue}.
By the Cauchy inequality and the contraction mapping principle, there exists $\widetilde{\epsilon}=\widetilde{\epsilon}(r)>0$ such that
for any $P\in \mathcal{H}_{r}(\widetilde{\epsilon})$  
 and $\xi\in B_{\C^n}(0,\frac{r}{4})$, we can find a unique $y_0=y_0(\xi)\in B_{\C^m} (0,\frac{r}{4})$ such that $$\nabla_{y} Z(\xi,y_0(\xi))=0,\quad \mbox{where}\quad Z(\xi,y) = {Q}(y) + P(\xi),\quad \forall\,(\xi,y) \in \R^n\times \R^m,$$ 
    then we know $Z(\cdot \, ,y_0(\cdot)) \in \mathcal{H}_{\frac{r}{4}}(\epsilon_{\frac{r}{4}})$ by the implicit function theorem (cf. e.g. \cite{Hor90}).
      
      Therefore, for any $\psi\in C_0^{\infty}(\R^{n+m})$ such that 
      \[
      \mbox{supp}\, \psi\subset {U_0}:=\left(B_{\C^n}\left(0,\frac{r}{4}\right)\cap A_r\right) \times B_{\R^m} \left(0,\frac{r}{4}\right),
      \]
       by the method of stationary phase (cf. \cite[Theorem 7.7.5]{HHor90}), there exist integer $k_0 \ge \frac{m}{2}-\beta$ and functions $\{\psi_j\}_{j=0}^{k_0-1}$ depending on the derivatives of $\psi$ with order up to $2k_0-2$ such that
 \begin{equation*}
       \int_{\R^m}e^{it Z(\xi,y)}\psi(\xi,y)\,dy = (2\pi i)^{\frac{m}{2}} \frac{t^{-\frac{m}{2}} e^{it Z(\xi,y_0(\xi))} }{\sqrt{\mbox{detHess}_y Z(\xi,y_0(\xi))}}  \sum_{j=0}^{k_0-1} t^{-j} \psi_j ( \xi) + \mathcal{R}(t), \ \ \forall\, t>0,
   \end{equation*}
   where $ \mathcal{R}(t) =  \mathcal{O}\left( t^{-k_0}\|\psi \|_{C^{2k_0}({U_0})}
\right) $. Here for two functions $f_1$ and $f_2$, we write $f_1=\mathcal{O}(f_2)$ if $|f_1(t)|\le C |f_2(t)|$ for some constant $C>0$ independent of $t$.
Thus, 
  \begin{align*}
      \int_{\R^{n+m}}&e^{it (h_2(\xi,y)+P(\xi,y)) }\psi(\xi,y)\,d\xi dy =\int_{\R^n}e^{it h_1(\xi)}\left(\int_{\R^m} e^{it Z(\xi,y)}\psi(\xi,y)\,dy\right)d\xi\\
      &= t^{-\frac{m}{2}} \sum_{j=0}^{k_0-1} t^{-j} \int_{\R^n} e^{it(h_1(\xi)+Z(\xi,y_0(\xi)))} \, \widetilde{\psi}_j(\xi)\,d\xi + \|\psi\|_{C^{2k_0}(U_0)}\,\mathcal{O}(t^{-k_0}),\ \forall\, t>0.
  \end{align*}
  Then our conclusion follows from the condition $M(h_1)\curlyeqprec (\beta,p)$ and the choice of $k_0$.\\
  \end{proof}



\begin{prop}\label{prop-x1x2x3}
The following assertions hold:
\begin{gather*}
     (a)\;M(\xi_1^{k+1}) \curlyeqprec \left(-\tfrac{1}{k+1},0\right),\quad\forall\, k\in\N;\\
     (b)\;M( \xi_1^2 \xi_2 - \xi_2^3) \curlyeqprec \left(-\tfrac{2}{3},0\right);
     \quad (c)\;M(\xi_1 \xi_2 \xi_3) \curlyeqprec (-1,1).
\end{gather*}
\end{prop}

\begin{proof}
The assertion $(a)$ can be proved by the Van der Corput lemma, see. e.g. \cite[Chapter 8]{S93}. For assertion $(b)$, it is the normal form of $D_4^-$ singularity, which is one of the stable singularities, see e.g. \cite[Table 4.3.2]{D74}. 

Now we prove $(c)$, which was firstly considered in \cite{K83}.  Let 
$$h_1(\xi)=\xi_1\xi_2\xi_3\in \mathcal{E}_{\alpha_1,3}\quad\mbox{with} \ \ \alpha_1=\left(\tfrac{1}{3},\,\tfrac{1}{3},\,\tfrac{1}{3}\right),\quad\mbox{while}\ \ \kappa:=h_1|_{\mathbb{S}^{2}},$$
where $\mathbb{S}^2$ is the standard unit sphere in $\R^3$. A direct computation shows that any $\theta\in \mathbb{S}^2$ satisfying $\kappa(\theta)=\mbox{d} \kappa |_\theta = 0$ is a nondegenerate critical point of $\kappa$, where $\mbox{d}\kappa$ is the differential of $\kappa$. Noting that Definition \ref{def-ue} carries over to real analytic manifolds, by Lemma \ref{lem-quad} we have
   $ M(\kappa,\theta) \curlyeqprec (-1, 0).$ 
Then by \cite[Theorem 1, Theorem 2(1a)]{K83}, it suffices to prove the following estimate
\begin{equation}\label{h2}
    M(h) \curlyeqprec (-1,1),\quad\mbox{with}\ \ h\in \Big\{\xi_{3}^2+\mbox{sym}^2\mathbb{Z}^2,\ \xi_{3}^2+\mbox{sym}^3\mathbb{Z}^2,
    \
    \xi_3^2-\xi_1^2\xi_2^2 \Big\},
\end{equation}
where $\mbox{sym}^2\mathbb{Z}^2$ and $\mbox{sym}^3\mathbb{Z}^2$ denote the set of nonzero binary quadratic forms and binary cubic forms (not analytically diffeomorphic to $\xi_1^3$ or $\xi_2^3$) in $(\xi_1,\xi_2)$, 
respectively.

By Lemma \ref{lem-quad}, we only need to consider the cases $\mbox{sym}^3\mathbb{Z}^2$ and $\xi_1^2\xi_2^2$.
The former case can be reduced to $\xi_1^2 \xi_2$ or $D_4^{\pm}$ singularities by the arguments in \cite[Page 85]{M21}. Moreover, it holds that
\begin{equation*}
    M(\xi_1^2 \xi_2)\curlyeqprec \left(-\tfrac{1}{2},0\right) \quad\mbox{and}\quad M(\xi_1^2 \xi_2^2)\curlyeqprec \left(-\tfrac{1}{2} , 1 \right)
\end{equation*}
by \cite[Theorem 1]{IS21} and \cite[Theorem 1.6]{RSK21}, respectively. Therefore, we have proved (\ref{h2}). Following the notations in \cite{K83}, let 
$(\beta_1,p_1)=(-1,1)$ and $(\beta_2,p_2)=(-1,0)$, then we finish the proof by \cite[Theorem 1]{K83}.\\
\end{proof}

\section{Proof of Theorem \ref{main-thm}}\label{sec-pf}

The following arguments work for all dimensions $d\ge 2$, we shall focus on the case $d=4$ in Section \ref{sssec-sim}.

\subsection{Classification of the critical points}\label{ssec-analysis}
In view of \eqref{equ-Im}, we first give a detailed analysis of \eqref{equ-inte}. We choose a nonnegative function $\eta\in C_0^{\infty}(\R^d)$ with support in $(-2\pi,2\pi)^d$, which is non-vanishing on some neighborhood of $\mathbb{T}^d$ such that
\begin{equation*}
    \sum_{x\in\Z^d} \eta (\xi+2\pi x)= 1,\quad\forall\,\xi\in\R^d.
\end{equation*}
Let $\omega$ be as in \eqref{equ-gree}, since it is periodic, we have
\begin{equation}\label{eq-per}
    I(v,t)=\frac{1}{(2\pi)^d}\sum_{x\in\Z^d}\int_{\mathbb{T}^d}  e^{it\phi(v,\xi)} \, \frac{\eta(\xi+2\pi x)}{\omega(\xi)}d\xi=\frac{1}{(2\pi)^d}\int_{\R^d}e^{it\phi(v,\xi)}\frac{\eta(\xi)}{\omega(\xi)}d\xi.
\end{equation}

Noting that both the phase and amplitude in \eqref{eq-per} have singularity at 0, a direct calculation gives
\begin{equation}\label{equ-nabla}
    \nabla\omega(\xi)=\frac{1}{\omega(\xi)}\big(\sin\xi_1, \, \cdots, \,\sin\xi_d\big),\quad \forall \,\xi=(\xi_1,\cdots,\xi_d)\in\mathbb{T}^d\backslash\{0\}.
\end{equation}

 For any $v\in\R^d$, we define 
\begin{equation*}
    \mathcal{C}_{v}:=\Big\{\xi \in \mathbb{T}^d\backslash \{0\}: \nabla_{\xi}\phi(v,\xi)=0\Big\},
\end{equation*}
where $\nabla_\xi \phi(v,\xi)$ is the gradient of $\phi$ in $\xi$.
We also define 
\begin{equation}\label{equ-VV}
    \mathbf{V}(\xi):= \left({\sum_{j=1}^d \sin^2 \xi_j}\right)\left({\sum_{j=1}^d\left(2 - 2\cos \xi_j\right)}\right)^{-1},  \quad\forall\,\xi\in\mathbb{T}^d\backslash\{0\}.
\end{equation}

If $\mathcal{C}_v\ne \emptyset$, by \eqref{equ-nabla} and the expression of $\phi$, for any $\xi\in\mathcal{C}_v$, it holds that $|v|^2 = \mathbf{V}(\xi)<1$.
Therefore, $\mathcal{C}_v=\emptyset$ for any $v\in B_{\R^d}(0,1)^c$. On the other hand, \cite[Proposition 5.6]{S96} gives that $\mathcal{C}_v\neq \emptyset$ for any $v\in B_{\R^d}(0,1)$. So if we define 
\begin{equation*}
    \Sigma_{k} := \Big\{\xi \in \mathbb{T}^d\backslash\{0\}:  \mbox{corank\,Hess}\,\omega(\xi)=k \Big\} \, \bigcap  \left(\bigcup_{v\in B_{\R^d}(0,1)}\,\mathcal{C}_v\right)\quad\mbox{for}\ \  k=0,\cdots,d,
\end{equation*} 
then by the fact $$\mbox{Hess}_{\xi}\phi (v,\xi)=-\mbox{Hess}\,\omega(\xi),\quad\forall\,(v,\xi)\in\R^d\times\R^d,$$ the set of degenerate critical points and the corresponding velocity set can be formulated as 
\begin{equation}\label{equ-SO}
\Sigma:=\bigcup_{k=1}^{d}\Sigma_{k}\quad\mbox{and}\quad
\Omega:=\bigcup_{k=1}^{d}\Omega_k,\quad \mbox{with}\ \ \Omega_k=\nabla\omega(\Sigma_k),\ \ k=0,\cdots,d.
\end{equation} 
 Here $\nabla \omega(U)$ (resp. $(\nabla\omega)^{-1}(U)$) is the image (resp. preimage) of $U$ under the map $\nabla \omega$, and for simplicity we do not specify the dependence of these notations on $d$.

Now we give the characterization of $\Sigma$ and some useful properties of the velocity sets. Without loss of generality, only the first quadrant $[0,\pi]^d$ is considered by symmetry.

\begin{lemma}\label{lem-dege}
    Let $d\geq 2$, then
    \begin{itemize}

     \item[(i)] 
 $\Sigma_d=\emptyset$, and $\Sigma_1=\Sigma_1'\cup \Sigma_1''$, where $\Sigma_1''$ is the set of $\xi$ with exactly two components equal to $\frac{\pi}{2}$, while
    \begin{align*}
     \Sigma_1'=\left\{\xi\in [0,\pi]^d\backslash\{0\}:\sum_{j=1}^d \left(\cos \xi_j + \sec \xi_j\right) = 2d,\ \xi_j\neq \frac{\pi}{2}\ \mbox{for all}\ j\right\}.
     \end{align*}
     
    \item[(ii)] If $d\geq 3$, for any $2\le j\le d-1$, $\Sigma_j$ consists of $\xi$ with exactly ($j+1$) components equal to $\frac{\pi}{2}$.

    \end{itemize}
\end{lemma}

\begin{proof}
    For any $(\xi,\lambda)\in (\mathbb{T}^d\backslash\{0\})\times \R$, a direct computation yields
   \[
\mathcal{D}(\lambda,\xi)= \frac{1}{\omega(\xi)^{3}}\left(\omega(\xi)^3\prod_{j=1}^d \left(\frac{\cos\xi_j}{\omega(\xi)}-\lambda\right)-\sum_{i=1}^d \sin^2\xi_i \prod_{j\ne i} \left(\frac{\cos\xi_j}{\omega(\xi)}-\lambda\right)\right),
\]
 where $\mathcal{D}(\lambda,\xi)=\mbox{det}(\mbox{Hess}\,\omega(\xi)-\lambda\mathbb{I}_d)$ and $\mathbb{I}_d$ is the identity matrix.

First, if $\xi=(\xi_1,\cdots,\xi_d)$ with  
$k$ components equal to $\frac{\pi}{2}$ for some $1\le k\le d$, we assume that $\xi_{j}=\frac{\pi}{2}$ for $1\le j\le k$ without loss of generality. If $k=d$, it is clear that $$\mathcal{D}(\lambda,\xi)=(-1)^{d}\lambda^{d-1}(\lambda+\omega(\xi)^{-3}\,d).$$
If $k<d$, then $\xi_l\ne \frac{\pi}{2}$ for $k+1\le l\le d$ and $\mathcal{D}(\lambda,\xi)=\lambda^{k-1}\mathcal{D}_1(\lambda,\xi)$, where
\[
\mathcal{D}_1(\lambda,\xi)=\frac{(-1)^{k}}{\omega(\xi)^{d-k}}\left( \lambda\omega(\xi)^2+\frac{k}{\omega(\xi)}-\lambda\sum_{j=k+1}^d \frac{\sin^2\xi_j}{\cos\xi_j-\lambda\omega(\xi)} \right)\prod_{j=k+1}^d\left( \cos\xi_j-\lambda\omega(\xi)\right),
\]
which makes sense near $\lambda=0$, and $\mathcal{D}_1(0,\xi)\ne 0$. 

 Next, if $\xi=(\xi_1,\cdots,\xi_d)$ with $\xi_j\ne \frac{\pi}{2}$ for all $j$, then we get $\mathcal{D}(0,\xi)=0$ if and only if the equation $\omega(\xi)^2=\sum_{j=1}^d\frac{\sin^2\xi_j}{\cos\xi_j}$ holds. In this case we have 
\[
\mathcal{D}(\lambda,\xi)=\lambda\, \mathcal{D}_2(\lambda,\xi),\quad\mbox{with}\ \ \mathcal{D}_2(\lambda,\xi)=- \frac{1}{\omega(\xi)^{d+1}}\sum_{i=1}^d\frac{\sin^2\xi_i}{\cos\xi_i}\prod_{j\ne i}\left( \cos\xi_j-\lambda\omega(\xi) \right).
\]
Then it is clear that $\mathcal{D}_2(0,\xi)\ne 0$. Collecting all these facts, we complete the proof.\\
\end{proof}

\begin{corollary}\label{cor-b0}
        Let $d\geq 3$, then
        \begin{equation}\label{asse-a}
            \Omega\subset B_{\R^d}(0,b_0)\quad \mbox{for some}\ \  b_0=b_0(d)\in (0,1).
        \end{equation}
        Moreover,
        \begin{equation}\label{asse-b}
            (\nabla \omega)^{-1}(\Omega_{d-1})=\Sigma_{d-1}\quad \mbox{and}\quad\Omega_i\cap\Omega_j=\emptyset,\quad \forall\, i,j\ge 2,\ i\neq j.
        \end{equation}
\end{corollary}

\begin{proof}
    To prove \eqref{asse-a}, by \eqref{equ-VV} we know
    \[
    \lim_{\xi\rightarrow 0}\mathbf{V}(\xi)=1\quad \mbox{and}\quad \mathbf{V}(\xi)<1,\ \ \forall\,\xi\in\mathbb{T}^d\backslash\{0\}.
    \]
    Then it suffices to prove $0\notin \overline{\Sigma}$, where we use $\overline{\Sigma}$ to denote the closure of $\Sigma$. In view of Lemma \ref{lem-dege}, we only need to show $0\notin \overline{\Sigma_1'}$. This can be deduced by the fact that
    \[
    \sum_{j=1}^{d} \left(\cos\xi_j+\sec\xi_j\right)-2d = \sum_{j=1}^d \frac{(\cos\xi_j-1)^2}{\cos\xi_j}=0,\quad\forall\,\xi\in\Sigma_1'.
    \]

    Now we prove \eqref{asse-b}. The first assertion is a direct consequence of \eqref{equ-nabla}. For the second assertion, since $\Sigma_d=\emptyset$ by Lemma \ref{lem-dege}, we argue by contradiction that $\Omega_i\cap\Omega_j\ne \emptyset$ with some $2\le i < j \le d-1$. Then there exist $\xi^*\in \Sigma_i$ and $\xi^{**}\in\Sigma_j$ such that
    \[
    \frac{1}{\omega(\xi^*)}\big(\sin\xi^*_1, \, \cdots, \,\sin\xi^*_d\big) = \frac{1}{\omega(\xi^{**})}\big(\sin\xi^{**}_1, \, \cdots, \,\sin\xi^{**}_d \big).
    \]
Without loss of generality, we assume $\xi^{*}_l=\frac{\pi}{2}$ for $1\le l\le i+1$ by Lemma \ref{lem-dege}, then we consider the following two cases separately. 

First, if there exists $1\le j_0\le i+1$ such that $\xi^{**}_{i_0}=\frac{\pi}{2}$, then $\omega(\xi^*)=\omega(\xi^{**})$. Since $j>i$, there exists $ j_1> i+1$ such that $\xi_{j_1}^{**}=\frac{\pi}{2}$. Then it holds that $\frac{1}{\omega(\xi^{**})}=\frac{\sin\xi_{j_1}^{*}}{\omega(\xi^*)}$, which is a contradiction.

Next, if $\xi_{l}^{**}\ne \frac{\pi}{2}$ for any $1\le l\le i+1 $, then we get $\omega(\xi^*)>\omega(\xi^{**})$ by the fact that $\frac{1}{\omega(\xi^{*})}=\frac{\sin\xi_{1}^{**}}{\omega(\xi^{**})}$. However, there exists $j_2>i+1$ such that $\xi_{j_2}^{**}=\frac{\pi}{2}$, by the same token we get $\omega(\xi^{**})>\omega(\xi^{*})$, a contradiction.    In conclusion, we finish the proof of \eqref{asse-b}.\\
\end{proof}

 By Corollary \ref{cor-b0}, for any 
$d\ge 2$, \emph{the most degenerate case} appears exactly at
\begin{equation}\label{equ-most-dege}
   (\xi,v)\in \Sigma_{d-1}\times \Omega_{d-1}=\left(\tfrac{\pi}{2},\cdots,\tfrac{\pi}{2}\right)\times\left(\tfrac{1}{\sqrt{2d}},\cdots,\tfrac{1}{\sqrt{2d}}\right)\in \R^d\times\R^d.
\end{equation}
\\

\subsection{Reductions}\label{sssec-decay} 
 As we mentioned before, outside the light cone there is no critical points, this case is simple. In fact, by Lemma \ref{lem-nocritical}, \eqref{equ-Im} and \cite[Proposition 5.3]{S96}, for any $\mathrm{r}>1$ and $N\in\N$, there exists $C=C(\mathrm{r},N)$ such that
 $$
 |G(tv,t)|\le C(1+|t|)^{-N},\quad \forall\,(v,t)\in B_{\R^d}(0,\mathrm{r})^c \times \R.
 $$

 However, in the vicinity of the light cone (cf. \eqref{regions}), we need more analysis. By \cite[Proposition 2.1, Proposition 2.2, Proposition 3.10]{S98}, we have:

\begin{lemma}\label{lem-schu}
    Let $d \geqslant 2$, then there exist $C=C(d)$ and $\mathbf{c}=\mathbf{c}(d)\in (b_0,1)$ such that 
    \begin{equation*}
        |G(tv,t)| \leqslant C (1+|t|)^{-\frac{d}{2}} ,\quad\forall\, (v,t)\in B_{\R^d}(0,\mathbf{c})^c \times \R.
    \end{equation*}
\end{lemma}

Its proof relies on the method of stationary phase, the properties of Airy function and the Green's function in the continuous setting. 

 By Lemma \ref{lem-schu} and \eqref{equ-Im}, it suffices to consider $I(v,t)$ for $v\in B_{\R^d}(0,\mathbf{c})$.

Since $\frac{1}{\omega}$ has singularity at 0, we choose $\chi\in C_0^{\infty}(\R^d)$ supported near the origin, then
\begin{align}\label{equ-I1+I2}
\begin{split}
    I(v,t)&=\frac{1}{(2\pi)^d}\int_{\R^d}e^{it\phi(v,\xi)}\frac{\eta(\xi)}{\omega(\xi)}\chi(\xi)\,d\xi+\frac{1}{(2\pi)^d}\int_{\R^d}e^{it\phi(v,\xi)}\frac{\eta(\xi)}{\omega(\xi)}\left(1-\chi(\xi)\right)d\xi\\
&=:I_1(v,t)+I_2(v,t). 
\end{split}   
\end{align}

By \cite[Proposition 2.3]{S98}, there exists $C=C(d)>0$ such that 
\begin{equation}\label{equ--I1}
    |I_1(v,t)|\le C|t|^{-d+1},\quad\forall\,(v,t)\in B_{\R^d}(0,\mathbf{c})\times \R.
\end{equation}

Now we consider $I_2$. Since 
$ \Omega\subset B_{\R^d}(0,\mathbf{c})\subset \Omega_0\cup\Omega$,
we have:
\begin{lemma}\label{lem-nondege}
    Let $d\ge 2$, for any $v_0\in \Omega_0\backslash\Omega$, there exist a neighbourhood $V$ of $v_0$ and $C>0$ such that
    \begin{equation*}
       |I_2(v,t)|\leq C(1+|t|)^{-\frac{d}{2}},\quad \forall\, (v,t) \in V \times \R.
    \end{equation*}
\end{lemma}

Lemma \ref{lem-nondege} follows from Lemma \ref{lem-quad} and a partition of unity in the space of wave number $\xi$. 
Therefore, we are left with the degenerate case. 
For any $v_0 \in \Omega$ and
$\xi\in\mbox{supp}\, \eta(1-\chi)=:\mathcal{U}$, assume for now that 
\begin{equation}\label{equ-Mxi}
    M\big(\phi(v_0,\cdot), \, \xi \big) \curlyeqprec (\beta_{\xi},p_{\xi})\quad\mbox{for some} \ \ (\beta_\xi,p_\xi)\in (-\infty,0)\times \N.
\end{equation} 
Then we can find a neighborhood $\mathbf{m}_{\xi}$ of $\xi$ as in Definition \ref{def-ue} such that $\mathcal{U}\subset\cup_{\xi\in\mathcal{U}}\mathbf{m}_{\xi}$ and \eqref{equ-uniform} holds. 
By a finite covering and a partition of unity, there are open sets, say $\{\mathbf{m}_{j}\}_{j=1}^{N_0}$, and nonnegative functions $\{\varphi_j\}_{ j=1}^{N_0}$ such that
\begin{equation*}
    \mathcal{U}\subset \bigcup_{j=1}^{N_0} \mathbf{m}_{j}\ \ \text{and}\ \  \sum_{j=1}^{N_0} \varphi_j  \equiv 1\  \ \mbox{on}\ \ \mathcal{U}, \quad \mbox{where} \ \ \varphi_j \in C_0^{\infty}(\mathbf{m}_{j}), \ j=1,\cdots,N_0.
\end{equation*}
Therefore,
\begin{equation*}
   (2\pi)^d\, I_2(v,t) = \sum_{j=1}^{N_0} \int_{\R^d} e^{it\phi(v, y)}\;\frac{\eta(y)(1-\chi(y))}{\omega(y)}\varphi_j(y)\,dy=:\sum_{j=1}^{N_0}I_2^{j}(v,t).
\end{equation*}

By \eqref{equ-Mxi}, Definition \ref{def-ue} and \eqref{equ-lin-per}, for any $1\le j\le N_0$, there exist  
$\epsilon_j$ and $C_j$ such that
\begin{equation*}
    |I_2^j(v,t)| \leqslant C_j(1+|t|)^{\beta_{\xi_j}}\log^{p_{\xi_j}}(2+|t|),\quad \forall\, (v,t)\in B_{\R^d}(v_0,\epsilon_j)\times \R,
\end{equation*}
and we obtain a similar result for $I_2$ by summing up these inequalities, with an exponent $(\beta,p)=\max_{1\le j\le N_0}\{(\beta_{\xi_j},p_{\xi_j})\}$ in the lexicographic order, i.e. $\beta=\max_{1\le j\le N_0} \beta_{\xi_j}$, and $p$ is the maximum among those $p_{\xi_j}$ such that $\beta=\beta_{\xi_j}$.

Now we establish \eqref{equ-Mxi}. If $\xi \notin \mathcal{C}_{v_0}$, we use Lemma \ref{lem-nocritical}. If $\xi\in \Sigma_0$, we use Lemma \ref{lem-nondege}. 
Finally, we consider the case $\xi\in\Sigma$ and associated $v_0\in\Omega$ (cf. \eqref{equ-SO}).

When $d=4$, it suffices to consider the following three cases by Corollary \ref{cor-b0}.

\begin{prop}\label{thm-pro}
    Let $d=4$, then
    \[
    M\big(\phi(v_0,\cdot),\, \xi_0\big) \curlyeqprec (\beta,p),\ \ \mbox{with}\ \   (\beta,p)= 
        \begin{cases}
           (-3/2,1), & \text{if $~(v_0,\xi_0) \in \Omega_3\times \Sigma_3$}; \\
            (-5/3,0), & \text{if $~(v_0,\xi_0) \in (\Omega_2\backslash\Omega_1)\times\Sigma_2$}; \\
            (-3/2,0), & \text{if $~(v_0,\xi_0) \in \Omega_1\times\Sigma_1$}.
        \end{cases}
    \]
\end{prop}

Once proving Proposition \ref{thm-pro}, we complete the proof of Theorem \ref{main-thm} by a finite covering on the velocity space $B_{\R^d}(0,\mathbf{c})$.
All cases in $d=2,3,4$ and the corresponding decay rates are listed in \hyperref[table-2]{Table 1} (cf.  \cite{S98} for the case $d=2,3$), we point out that they can also be interpreted by Newton polyhedra, see Section \ref{ssec-alt234}. \\

\begin{table}[h]
\renewcommand\arraystretch{1.5}
    \begin{center}
        \caption{Disperive estimates for the DW}\label{table-2}
        \begin{tabular}{|c|c|c|c|c|}
             \hline
              & $\Sigma_1'$ & $\Sigma_1''$ & $\Sigma_2$ & $\Sigma_3$ \\ 
             \hline
             dim 2 & $|t|^{-5/6}$ & $|t|^{-3/4}$ &  &  \\
             \hline
             dim 3 & $|t|^{-4/3}$ & $|t|^{-5/4}$ & $|t|^{-7/6}$ &  \\
             \hline
             dim 4 & $|t|^{-3/2}$ & $|t|^{-3/2}$ & $|t|^{-5/3}$ & $|t|^{-3/2}\log(|t|)$\\
             \hline
        \end{tabular}
    \end{center}
\end{table}

\subsection{The most degenerate case}\label{ssec-most}
Before proving Proposition \ref{thm-pro},
we notice that the decay rate is determined by the most degenerate case (cf. \eqref{equ-most-dege})
    when $d= 3, 4$. We give a  decomposition of the phase in such cases on $\Z^d$ for all $d\ge 3$, which will also be used in the proof of Theorem \ref{thm-conj}, see Section \ref{ssec-odd}.
    
    In the sequel, let $\{\boldsymbol{e_{j}}\}_{j=1}^d$ be the standard coordinate vector in $\R^d$, we define a weight
    \begin{equation}\label{def-weight}
        \boldsymbol{w_d}:=\tfrac{1}{3}(\boldsymbol{e_{1}}+\cdots +\boldsymbol{e_{d-1}})+\tfrac{1}{2}\boldsymbol{e_{d}}\in\R^d.  
    \end{equation}
    This choice comes from the principle face of the Newton polyhedra, see Section \ref{sec-new}.  For any $(a,b)\in\R^2$ and $m\ge 1$, we set
    \begin{equation}\label{equ-def-Q}
        \mathbf{Q}_{a,b}^{m}(z):=a\left(\sum_{j=1}^{m}z_j\right)^3-b\sum_{j=1}^{m} z_j^3, \quad \forall\, z\in\R^{m}.
    \end{equation}

\begin{lemma}\label{lem-expand}
    Let $d\geq 3$, $\phi$ be as in \eqref{equ-inte} and $(v_0,\xi_0)\in \Omega_{d-1}\times \Sigma_{d-1}$, then there exist an invertible linear transform $\Phi$ on $\R^d$, a constant $c_\phi=c_{\phi}(d)$ and $R\in H_{\boldsymbol{w_d},d}$ such that 
    \begin{equation*}
        \phi(v,\Phi(y)+\xi_0) = c_\phi+v\cdot\xi_0+\Phi(y)\cdot(v-v_0)+y_d^2+\mathbf{Q}_{1,1}^{d-1}(y')+R(y)
    \end{equation*}
       holds for $v\in\R^d$ and $y$ near the origin, where $y'=(y_1,\cdots,y_{d-1})$. 
\end{lemma}
 
\begin{proof}
   By the Taylor's formula of $\phi$ at $\xi_0$, for some 
$a_j>0$, $1 \le j \le 3$, we have
    \begin{equation*}
        \phi(v,\xi+\xi_0)= c_\phi+v\cdot\xi_0+(v-v_0)\cdot\xi+a_1\left(\sum_{j=1}^d \xi_j\right)^2-a_2\left(\sum_{j=1}^d \xi_j\right)^3+a_3\sum_{j=1}^d \xi_j^3+W(\xi),
    \end{equation*}
     with
     $
     W(\xi) = \sum_{k=4}^\infty  W_k(\xi)$ and $W_k\in \mbox{span}_{\R} \mathbf{W}_k$, where for any $k\ge 4$,
     \[
      \mathbf{W}_k := \left\{\prod_{l=1}^d\left(\sum_{j=1}^d\xi_j^{\,q_l}\right)^{i_l} : \ \sum_{l=1}^d q_l\, i_l=k, \ q_l,i_l\in\N, \ q_l\ \mbox{odd} 
     \right\}
     \]
     is a subset of homogeneous  polynomial of degree $k$. 
    Then a change of coordinates $$\xi=\widetilde{\Phi}(z),\ \ \mbox{with}\ \ z_j=\xi_j,  \ j=1,\cdots,d-1,\ \ \mbox{and}\ \  z_d=\sum_{j=1}^d \xi_j$$
    gives that
    \begin{align}\label{eq-lin}
        \phi(v,\widetilde{\Phi}(z)+\xi_0)=c_\phi+v\cdot\xi_0+ (v-v_0)\cdot \widetilde{\Phi}(z) + a_1  z_d^2+a_3\,\mathbf{Q}_{1,1}^{d-1}(z')+\widetilde{W}(z),
    \end{align}
    where $z'=(z_1,\cdots,z_{d-1})$ and 
    \begin{equation*}
    \widetilde{W}(z)=-a_2\, z_d^3-3a_3\, 
 z_d^2\sum_{j=1}^{d-1}z_j+3a_3 \,z_d\sum_{j=1}^{d-1}z_j^2+W\circ\widetilde{\Phi}(z).
    \end{equation*}
    One can easily check that $W\circ\widetilde{\Phi}\in H_{\boldsymbol{w_d},d}$ and thus  $\widetilde{W}\in H_{\boldsymbol{w_d},d}$.
    Then we get the conclusion by absorbing  coefficients in (\ref{eq-lin}).\\
\end{proof}

\begin{remark}\label{rem-per}
     For any given $\epsilon,\, r>0$,  the ``perturbation part" $y \mapsto \Phi(y)\cdot(v-v_0)$ is linear, and it belongs to $\mathcal{H}_{r}(\epsilon)$ if $v$ is close to $v_0$. 
    Besides, by the splitting lemma (cf. e.g. \cite[Theorem 4.13]{M21}) we may obtain a result similar to Lemma \ref{lem-expand}, but the ``perturbation part" is not linear any more.\\
\end{remark}

\subsection{Proof of Proposition \ref{thm-pro}}\label{sssec-sim}

Let $d=4$. By Lemma \ref{lem-dege} and the symmetry, we can assume
\begin{gather*}
\begin{split}
    \Sigma_3 =\left(\tfrac{\pi}{2},\tfrac{\pi}{2},\tfrac{\pi}{2},\tfrac{\pi}{2}\right),\quad\Sigma_2  =\left\{
        \left(\tfrac{\pi}{2},\tfrac{\pi}{2},\tfrac{\pi}{2},\xi_*\right): \xi_* \neq \tfrac{\pi}{2} \right\}\\
    \mbox{and} \ \ \Sigma_1 =\Sigma_1'\cup\left\{
        \left(\tfrac{\pi}{2},\tfrac{\pi}{2},\xi_*,\eta_*\right): \xi_*,\eta_* \neq \tfrac{\pi}{2} \right\}.
\end{split}
\end{gather*}

We consider the three cases separately.

\subsubsection{$(v_0,\xi_0)\in\Omega_3\times\Sigma_3$}

In this case,  by Lemma \ref{lem-expand} we get
\begin{align*}
    \phi(v,\Phi(y)+\xi_0)=c_\phi+v\cdot\xi_0+y_4^2+\Phi(y)\cdot(v-v_0)+\mathbf{Q}_{1,1}^3(y')+R(y),
\end{align*}    
where $R\in H_{\boldsymbol{w_4},4}$.  Noticing that 
$$\mathbf{Q}_{1,1}^3(y')= (y_1+y_2+y_3)^3-y_1^3-y_2^3-y_3^3=3(y_1+y_2)(y_1+y_3)(y_2+y_3),$$ 
a change of coordinates gives
\begin{align*}
    \phi_1 = c_\phi+v\cdot\xi_0+z_4^2+z_1z_2z_3+{\Phi_1}(z)\cdot(v-v_0)+{R_1}(z),
\end{align*}
where the means of $\phi_1$, $\Phi_1$ and $R_1$ are obvious.
Moreover, we have 
\begin{align*}
   M\left(z_4^2 +3z_1z_2z_3+{R_1}\right)\curlyeqprec M\left(z_4^2 +3z_1z_2z_3\right)\curlyeqprec M\left(z_1z_2z_3\right)+\left(-\frac{1}{2},0\right)\curlyeqprec\left(-\frac{3}{2},1\right)
\end{align*}
by Lemma \ref{lem-prin}, Lemma \ref{lem-quad} and Proposition \ref{prop-x1x2x3} (c). Then our conclusion follows from Remark \ref{rem-per}.

\subsubsection{$(v_0,\xi_0)\in(\Omega_2\backslash\Omega_1)\times\Sigma_2$}\label{ssec-co2}

In this case, $\xi_0 = \left(\frac{\pi}{2},\frac{\pi}{2},\frac{\pi}{2},\xi_*\right)^T$ with $\xi_* \neq \frac{\pi}{2}$. A direct computation shows that the zero-eigenvectors of $\mbox{Hess}_\xi \phi(v_0,\xi_0)$ are
$$\boldsymbol{\gamma_1}=(1,-1,0,0)^T\quad\mbox{and}\quad\boldsymbol{\gamma_2}=(1,1,-2,0)^T.$$ 
 Then we have, for some constant $c$, 
 \begin{equation*}
    \begin{aligned}
        \phi(v, \mathbb{A} y+ \xi_0)= & \, c+(v-v_0)\cdot\mathbb{A} y +\frac{\sqrt{2}}{2}\omega(\xi_0)^{-1} (y_2^3 - y_1^2 y_2 )+ \mathcal{V}(y)\\
        &- \frac{\sqrt{2}}{8}\omega(\xi_0)^{-3}\Big(y_3^2 + 2 y_3\, y_4 \sin \xi_* - \left(2\omega(\xi_0)^2\cos \xi_* - \sin ^2 \xi_*\right) y_4^2\Big),
    \end{aligned}
\end{equation*}
 where the matrix
  $
  \mathbb{A}=(\boldsymbol{\gamma_1},\boldsymbol{\gamma_2},\boldsymbol{e_{3}},\boldsymbol{e_{4}}).$
Moreover,
 $\mathcal{V}\in H_{\alpha_*,4}$ with $\alpha_*=\left(\frac{1}{3},\frac{1}{3},\frac{1}{2},\frac{1}{2}\right)$. 
 A change of coordinates in $(y_3,y_4)$  gives that  
\begin{equation*}
    \phi = c + (v-v_0)\cdot \widetilde{\mathbb{A}}y +\Big[a_4\, y_3^2 + a_5 y_4^2 - y_1^2y_2 + y_2^3 \Big] + \widetilde{\mathcal{V}}(y), \quad \mbox{with some} \ \ a_4,a_5\ne 0.
\end{equation*}
The polynomial in the square bracket is the normal form of $D_4^-$ singularity, then Lemma \ref{lem-quad} and Proposition \ref{prop-x1x2x3} (b) give that
\begin{align*}
    M\left( a_4 y_3^2 + a_5 y_4^2- y_1^2y_2 + y_2^3  \right)\curlyeqprec M\left( - y_1^2y_2 + y_2^3  \right)+(-1,0)\curlyeqprec \left(-\frac{5}{3},0\right),
\end{align*}
and we finish the proof by Remark \ref{rem-per} again.

\subsubsection{$(v_0,\xi_0)\in\Omega_1\times\Sigma_1$}\label{ssec-co3}
In this case, a combination of the splitting lemma \cite[Theorem 4.13]{M21} and Lemma \ref{lem-quad} directly gives a uniform estimate with exponent $(-\frac{3}{2},0)$, which is better than the first case. 

As a consequence, the proof of Proposition \ref{thm-pro} is completed.\\

\section{Newton Polyhedra}\label{sec-new}
\subsection{Basic concepts and results}\label{ssec-new}

We recall some basics of Newton polyhedra, see also \cite{AGV12,IM11,V76}. For the concepts in convex analysis and polytopes, we refer to \cite{BJ77,AB83}.

Let $S$ be a function on $\R^d$ and real-analytic at 0, we shall always assume that 
\begin{equation}\label{con-phi}
    S(0)=0\quad\mbox{and}\quad \nabla S(0)=0.
\end{equation}
Consider the associated Taylor series at 0,
\begin{equation}\label{equ-series}
    S(\xi) = \sum_{\gamma\in \mathbb{N}^d} s_{\gamma} \, \xi^{\gamma}.
\end{equation}
The set 
$\mathcal{T}(S) := \{\gamma \in \mathbb{N}^d :s_{\gamma} \neq 0 \}$ is called the {Taylor support}. The \emph{Newton polyhedron} $\mathcal{N}(S)$ is the convex hull of the set
\begin{equation*}
    \bigcup_{\gamma \, \in\, \mathcal{T}(S)} \left(\gamma + \mathbb{R}^d_{+}\right), \quad\mbox{where}\ \ \R^d_+=\{\xi\in \R^d: \xi_j\geq 0, \, j=1,\cdots,d \}.
\end{equation*}
 
Let $\mathcal{P}$ be a face  of $\mathcal{N}(S)$, we call 
$
    S_{\mathcal{P}}(\xi) := \sum_{\gamma \in \mathcal{P}}s_{\gamma}\, \xi^{\gamma} 
$
the $\mathcal{P}$-part of the series in (\ref{equ-series}). Moreover, we say $S$ is $\mathbb{R}$-nondegenerate if for any compact face $\mathcal{P}$, $\nabla S_{\mathcal{P}}$ is nonvanishing on $(\mathbb{R}\backslash\{0\})^d$, that is,
\begin{equation}\label{equ-nondegenerate}
    \bigcap_{j=1}^d \big\{\xi : {\partial_{j}S_{\mathcal{P}}}(\xi)=0\big\}\ \subset\ \, \bigcup_{j=1}^d \big\{\xi : \xi_j=0 \big\}.
\end{equation}

If $\mathcal{T}(S)\ne \emptyset$, the {Newton distance} $d_S$ and the center $\boldsymbol{d}_S$ are defined as
\begin{equation*}
        d_S = \inf \{\varrho>0:(\varrho,\varrho,\cdots,\varrho) \in \mathcal{N}(S)\} \quad \mbox{and} \quad \boldsymbol{d}_S:=(d_S,d_S,\cdots,d_S)\in\R^d.
\end{equation*}
 The principle face $\pi_S$ of $\mathcal{N}(S)$ is the face of minimal dimension containing 
$\boldsymbol{d}_S$ and $k_S:=d-\dim_{\R^d} (\pi_S)$, where $\dim_{\R^d}(U)$ is the affine dimension of $U$ in $\R^d$. The $\pi_S$-part, denoted by $S_{\pi}$, is called the principle part of $S$. 

Since $d_S$ depends on the choice of the coordinate systems, the {height} of $S$ is given by
\begin{equation*}
    h_S:=\sup \{d_{S,\xi}\},
\end{equation*}
where the supremum is taken over all local analytic coordinate systems $\xi$ which preserves the origin, and $d_{S,\xi}$ is the Newton distance in coordinates $\xi$. A given coordinate system $\xi_a$ is said to be {adapted} to $S$ if $d_{S,\xi_a} = h_S$. 

In our setting, the following result, derived by \cite[Proposition 0.7, 0.8]{V76}, can be used to recognize whether a given coordinate system is adapted.
\begin{prop}\label{prop-?}
Let $d=2$, if $\boldsymbol{d}_S$ lies on a compact face $\Gamma$ of $\mathcal{N}(S)$ with $\Gamma\subset \{\xi: a_1 \xi_1 + \xi_2 = a_2\}$ for some $a_1,a_2\in\N$, 
then the coordinate system is adapted if $S_{\Gamma}(\cdot\,,1)$ does not have a real root of multiplicity larger than $\frac{a_2}{1+a_1}$.
\end{prop}

Now let $S$ be as in \eqref{con-phi} and $J(t,S,\psi)$ be as in \eqref{equ-OI}, 
where $\psi\in C_0^\infty(\R^d)$ with support near the origin. Then the following asymptotic expansion holds (cf. e.g. \cite[pp. 181]{AGV12}),
\begin{equation}\label{equ-expan}
    J(t,S,\psi) \approx \sum_{\tau} \sum_{\rho=0}^{d-1} c_{\tau,\rho,\psi} \, t^\tau \log^\rho t,\quad \mbox{as} \ \ t\rightarrow+\infty,
\end{equation}
where $\tau$ runs through finitely many arithmetic progressions not depending on $\psi$, which consists of negative rational numbers.

Let $(\tau_S,\rho_S)$ be the maximum over all pairs $(\tau,\rho)$ in \eqref{equ-expan} under the
lexicographic ordering such that for any neighborhood $U$ of the origin, there exists $\psi\in C_0^{\infty}(U)$  for which $c_{\tau_S,\rho_S,\psi}\ne 0$. We call $\tau_S$ the oscillation index of $S$ at 0 and $\rho_S$ its multiplicity.

The following useful result is a consequence of the main theorem in Varchenko \cite{V76},
see also Gilula \cite[Theorem 2.3]{G18}.

\begin{theorem}\label{thm-new}
    Let $S$ be $\mathbb{R}$-nondegenerate and \eqref{con-phi} hold,  then $\tau_S \le -{d_S}^{-1}$ and $\rho_S \le k_S-1$, that is, for any $\psi\in C_0^\infty(\R^d)$ with support near the origin, there exists $C>0$ such that
    \begin{equation*}
       |J(t,S,\psi)|\leqslant C (1+|t|)^{-\frac{1}{d_S}} \log^{k_S - 1}(2+|t|), \ \ \forall\, t\in \R.
    \end{equation*}
    Moreover, if $d_S>1$, then $\tau_S=-d_S^{-1}$.
\end{theorem}

If $d=2$, \cite[Theorem 2.1]{K84} and \cite[Theorem 0.6]{V76} give the following stronger result:
\begin{theorem}\label{thm-twodim}
    Let $S:\mathbb{R}^2 \rightarrow \mathbb{R}$ be as in (\ref{con-phi}),  then there exist coordinate systems that are adapted to $S$.
    Moreover,  $
    M(S)\curlyeqprec (\tau_S,\rho_S)
    $ and $
    \tau_S=-h_S^{-1}.
    $
    \\
\end{theorem}

\subsection{Application to the DW}\label{ssec-alt234}

The decay rates in \hyperref[table-2]{Table 1} can be interpreted by the Newton polyhedra.
First, let $\phi$ be as in \eqref{equ-inte}. At its critical point the principle part can be expressed in the form $\phi_{\pi}=\phi_1+q_\phi$  under suitable coordinate systems, where $q_\phi$ is the quadratic term separating from other variables, see  \hyperref[table-3]{Table 2}. For the details, we refer to \cite{S98} for the case $d=2,3$ and Section \ref{sssec-sim} for the case $d=4$.

\begin{table}[h]
\renewcommand\arraystretch{1.5}
    \begin{center}
        \caption{Cases in $d=$ 2,3,4}\label{table-3}
        \begin{tabular}{|c|c|c|c|c|c|}
             \hline
               &  
            & $\phi_\pi=\phi_1(+q_{\phi})$ & $\mathcal{N}(\phi_1)$& $d_{\phi_1}$ & $(\beta_{\phi_\pi},p_{\phi_\pi})$ \\ 
            \hline
             dim 2 & $\Sigma_1'$ 
            & $\xi_1^3  (+ \xi_2^2)$ & \multirow{2}{*}{$[3,\infty)$} & \multirow{2}{*}{3} & $(-5/6,0)$\\
             \cline{1-3}\cline{6-6}
             dim 3 & $\Sigma_1'$ 
            & $\xi_1^3 (+\xi_2^2+ \xi_3^2)$ &  &  & $(-4/3,0)$\\ 
             \hline
             dim 2 & $\Sigma_1''$ 
            & 
            $\xi_1^2 + \xi_1\xi_2^2$ & \multirow{2}{*}{$\{(\lambda+1,2-2\lambda):\lambda \in [0,1]\}+\R^2_+$} & \multirow{2}{*}{$\dfrac{4}{3}$} & $(-3/4,0)$\\
            \cline{1-3}\cline{6-6}
              dim 3 & $\Sigma_1''$ 
             & $\xi_1^2 + \xi_1\xi_2^2 (+\xi_3^2)$ &  &  & $(-5/4,0)$\\
             \hline
              dim 3 & $\Sigma_2$ 
             & $\xi_1^2\xi_2-\xi_2^3(+\xi_3^2)$ & \multirow{2}{*}{$\{(2\lambda,3-2\lambda):\lambda \in [0,1]\}+\R^2_+$} & \multirow{2}{*}{$\dfrac{3}{2}$} & $(-7/6,0)$\\
             \cline{1-3}\cline{6-6}
              dim 4 & $\Sigma_2 $ 
             &$\xi_1^2\xi_2-\xi_2^3 (+\xi
             _3^2+\xi_4^2)$ &  &  & $(-5/3,0)$\\
             \hline
              dim 4 & $\Sigma_3$ 
             &$\xi_1\xi_2\xi_3 (+\xi_4^2)$ & $(1,1,1) + \R^3_+$ & 1 & $(-3/2,1)$\\
             \hline
        \end{tabular}
    \end{center}
\end{table}
If we translate the critical point of $\phi$ to the origin, then we get $M(\phi)\curlyeqprec M(\phi_\pi)$ by Lemma \ref{lem-prin}. Moreover, we have: 
\begin{prop}
    For each case in \hyperref[table-3]{Table 2},  it holds that 
 $$M(\phi_\pi)\curlyeqprec (\beta_{\phi_\pi},p_{\phi_\pi})=(\tau_{\phi_\pi},\rho_{\phi_\pi}).$$
\end{prop}
\begin{proof}
    By Lemma \ref{lem-quad}, it suffices to consider the phase $\phi_1$. The case $\xi_1^3$ is handled by the Van der Corput lemma, while the case $\xi_1\xi_2\xi_3$ is proved by Proposition \ref{prop-x1x2x3} (c). 
    
    For the left two cases 
    $$\phi_1(\xi)= \xi_1^2+\xi_1\xi_2^2\quad \mbox{or}\quad \phi_1(\xi)=\xi_1^2\xi_2-\xi_2^3,$$  Proposition \ref{prop-?} shows that they are expressed in adapted coordinate systems. In fact, the supporting line of $(\phi_1)_{\pi}$ is 
    $$\{\xi: 2\xi_1+\xi_2=4\}\quad\mbox{or}\quad \{\xi: \xi_1+\xi_2=3\},\ \ \mbox{respectively}.$$
    In each case $(\phi_1)_\pi(\cdot\,,1)=\phi_1(\cdot\,,1)$ has real root of multiplicity $1$. Then using Theorem \ref{thm-twodim} and Theorem \ref{thm-new}, we know $$M(\phi_1)\curlyeqprec (-d_{\phi_1}^{-1},k_{\phi_1}-1)=(-d_{\phi_1}^{-1},0).$$ 
    The proof is completed. \\
\end{proof}


\begin{remark}
$(1)$  For the case $\xi_1\xi_2\xi_3$, we get a decay rate of order $|t|^{-1}\log^2 |t|$ directly by Theorem \ref{thm-new}. However, a change of coordinates $$\xi_1=z_1,\, \xi_2=z_2-z_3,\, \xi_3=z_2+z_3$$
gives a new phase $z_1(z_2^2-z_3^2)$, the associated decay rate is of order $|t|^{-1}\log |t|$ by Theorem \ref{thm-new} again. It is sharp, see Remark \ref{remark-main}.   $(2)$ As mentioned in Remark \ref{rem-per}, the perturbation part is linear. Then the result \cite[Theorem 1.1]{IM11} can also be applied to get the desired exponents, see e.g. \cite{BG17}. 
\end{remark}

We shall give the proof of Theorem \ref{thm-conj} in the following part.\\

\subsection{Proof of Theorem \ref{thm-conj}}\label{ssec-odd}
In view of \eqref{equ--I1} and \eqref{equ-I1+I2}, it suffices to consider $I_2(v_0,t)$.
By Corollary \ref{cor-b0} we know the unique critical point of $\phi(v_0,\cdot)$ is $\xi_0=(\frac{\pi}{2},\cdots,\frac{\pi}{2})$. 

Let $\phi_{0}(\xi):= \phi(v_0,\xi+\xi_0)$, we choose proper coordinate systems to turn $\phi_0$ into an $\R$-nondegenerate phase, we show that the associated Newton distance is $\frac{6}{2d+1}$ and the dimension of the principle face is $d-1$. Then our conclusion follows from Theorem \ref{thm-new}.\\
  
\textbf{Step 1.} Note that Hess$\,\phi_{0}(0)$ has rank $1$. By Lemma \ref{lem-expand}, there exist a linear transform $\Phi$, a constant $c_\phi$ and $R\in H_{\boldsymbol{w_d},d}$ such that
   \begin{equation*}
        \phi_0\circ\Phi(y) = c_\phi+y_d^2+\mathbf{Q}_{1,1}^{d-1}(y')+R(y).
    \end{equation*}
    
    Suppose that $$d=2\mathbf{k}+1\quad\mbox{with \ some}\quad\mathbf{k}\geq 1.$$ And let $y=\Psi(z)$, with
    \begin{equation*}
         \begin{cases}
           z_i=\dfrac{y_i+y_{i+1}}{2}, \\
            z_{i+1}=\dfrac{y_i-y_{i+1}}{2}, \quad 1\leq i\leq d-2,\quad i~\mbox{odd}, \\
            z_{d}=y_{d},
        \end{cases}
    \end{equation*}
then we get that
 \begin{equation}\label{eq-bian}
        \phi_0\circ\Phi\circ \Psi(z)=c_\phi+z_{d}^2+2Y(z')+R\circ \Psi(z)=:c_\phi+\widetilde{\phi}(z),
    \end{equation}
    where $z'=(z_1,\cdots,z_{d-1})$ and
    \begin{equation}\label{eq-R}
        Y(z')=4\left(\sum_{j=1}^{\mathbf{k}} z_{2j-1}\right)^3-\left(\sum_{j=1}^{\mathbf{k}} z_{2j-1}^3\right)-3\sum_{j=1}^{\mathbf{k}}z_{2j-1}z_{2j}^2.
    \end{equation}

   \textbf{Step 2.} Now we show that $\widetilde{\phi}$ is $\R$-nondegenerate. Let $\mathbf{S}=\widetilde{\phi}-R\circ \Psi$. We shall prove that $\mathbf{S}$ is $\R$-nondegenerate, and for any compact face $\mathcal{W}$ of $\mathcal{N}(\widetilde{\phi})$, it holds that $\widetilde{\phi}_{\mathcal{W}}=\mathbf{S}_{\mathcal{W}}$. Then $\widetilde{\phi}$ is $\R$-nondegenerate.

 Firstly, note that the quadratic term $z_d^2$ and the terms with the even index only appear once in (\ref{eq-R}).
  For any compact face $\Gamma$ on $\mathcal{N}(\mathbf{S})$, the $\Gamma$-part $\mathbf{S}_{\Gamma}$ containing the term $z_{2j-1}z_{2j}^2$ for some $1\le j\le \mathbf{k}$ or the term $z_d^2$ must be $\R$-nondegenerate, since $\partial_{2j} \mathbf{S}_{\Gamma}(z)=2z_{2j-1}z_{2j}$ and (\ref{equ-nondegenerate}) is fulfilled.

Then we reduce the problem to consider the left terms in $Y$, which is the polynomial $ \mathbf{Q}_{4,1}^{\mathbf{k}}$ (recall \eqref{equ-def-Q}) up to a renumbering of the coordinates. Now we show that $\mathbf{Q}_{4,1}^{\mathbf{k}}$ is $\R$-nondegenerate.

For any $k\in \N$ and $\gamma\in \R_+^k$, we define the $(k-1)$-simplex 
    \begin{equation*}
        \Theta_{k-1}^{\gamma}:=\left\{w\in \R^{k}_+:w\cdot \gamma=1\right\}.
    \end{equation*}
Notice that $\mathbf{Q}_{4,1}^{\mathbf{k}}$ is related to the compact fact $\Theta_{\mathbf{k}-1}^{\gamma_1}$, where $\gamma_{1}=(\tfrac{1}{3},\cdots,\tfrac{1}{3})$. We begin with $\mathbf{Q}_{4,1}^{\mathbf{k}}$ itself. Since $\nabla \mathbf{Q}_{4,1}^{\mathbf{k}}(w)=0$ is equivalent to the equation
   \begin{equation*}
4\left(\sum_{j=1}^{\mathbf{k}} w_j\right)^2=w_1^2=w_2^2=\cdots=w_{\mathbf{k}}^2,
    \end{equation*}
    which has only zero solution by a direct computation, then it suffices to consider other compact faces with dimension less than $\mathbf{k}-1$. 
    We claim that for any such face $\mathcal{P}$, the $\mathcal{P}$-part $(\mathbf{Q}_{4,1}^{\mathbf{k}})_\mathcal{P}$ has similar expression as $\mathbf{Q}_{4,1}^{\mathbf{k}}$. 

    Indeed, by the representation theorem (cf. e.g. \cite[Page 68]{BJ77}), we have
\begin{equation*}
         \mathcal{N}(\mathbf{Q}_{4,1}^{\mathbf{k}})=\Theta_{\mathbf{k}-1}^{\gamma_1}+\R_+^{\mathbf{k}}.
    \end{equation*} 
   By \cite[Theorem 12.1]{AB83}, we know that any vertex of $\mathcal{P}$ is also a vertex of $\Theta_{\mathbf{k}-1}^{\gamma_1}$, the vertices of 
   which are 
   $\{3{e_j}\}_{j=1}^{\mathbf{k}}$. Here $e_j$ ($1\le j\le \mathbf{k}$) denotes the standard coordinate vector in $\R^{\mathbf{k}}$.  
   Therefore, without loss of generality we can assume that the vertices of $\mathcal{P}$ is $\{3{e_j}\}_{j=1}^n$ for some $1\leq n< \mathbf{k}$. Then it is clear that $(\mathbf{Q}_{4,1}^{\mathbf{k}})_\mathcal{P}= \mathbf{Q}_{4,1}^{n}$, and by induction we know $\mathbf{Q}_{4,1}^{\mathbf{k}}$ is $\R$-nondegenerate, so is $\mathbf{S}$.

    Now we prove that for every compact face $\mathcal{W}$ of $\mathcal{N}(\widetilde{\phi})$, it holds that $\widetilde{\phi}_{\mathcal{W}}=\mathbf{S}_{\mathcal{W}}$. By (\ref{eq-bian}) and a direct computation, it suffices to prove that no vertex of $\mathcal{N}(\widetilde{\phi})$ comes from the terms 
   $T_1=z_{d}(\sum_{j=1}^{\mathbf{k}} z_{2j-1})^2$ 
   and
   \begin{align*}
       T_{2N+1}&= \sum_{j=1}^{\mathbf{k}} \Big((z_{2j-1}-z_{2j})^{2N+1} + (z_{2j-1}+z_{2j})^{2N+1}\Big)=\sum_{j=1}^{\mathbf{k}}\sum_{m=1}^{N+1}c_{m,N}\,z_{2j-1}^{2m-1}\,z_{2j}^{2(N-m+1)}
   \end{align*}
for any $N\ge 2$, where we also used the expression of $R\circ \Psi$ as in the proof of Lemma \ref{lem-expand}.
   One easily sees that $\mathcal{T}(T_{2N+1})\subset \mbox{ri}\left(\mathcal{N}(Y)\right)$, i.e. the relative interior of $\mathcal{N}(Y)$.
     As for $T_1$,
   we consider an unbounded face of $\mathcal{N}(\mathbf{S})$,
   \begin{align*}
       \mathcal{U}_\mathbf{S}&:=\{\xi\in\R_{+}^{d}:\xi_{2j}=0,\,1\leq j\leq \mathbf{k}\}\cap \mathcal{N}(\mathbf{S})\\
       &=\left\{\xi\in\R_{+}^{d}:\frac{\xi_{d}}{2}+\sum_{j=1}^{\mathbf{k}}\frac{\xi_{2j-1}}{3}\geq 1,\,\xi_{2i}=0,\,1\leq i\leq \mathbf{k}\right\}.
   \end{align*}
   Then it is clear that $\mathcal{T}(T_1)\subset \mbox{ri}(\mathcal{U}_\mathbf{S})$. 
   Therefore, the term $R\circ \Psi$ is negligible, which means that $\widetilde{\phi}_{\mathcal{W}}=\mathbf{S}_{\mathcal{W}}$ for any compact face $\mathcal{W}$ of $\mathcal{N}(\widetilde{\phi})$, so 
 $\widetilde{\phi}$ is $\R$-nondegenerate.\\

   \textbf{Step 3.}
    Now we prove $d_\mathbf{S}=\frac{6}{2d+1}$ and $k_\mathbf{S}=1$. Let $$A_{2j}=\boldsymbol{e_{2j-1}}+2\boldsymbol{e_{2j}},\quad A_{2j-1}=3\boldsymbol{e_{2j-1}}\quad\mbox{for}\ \ 1\le j\le \mathbf{k},\quad\mbox{and}\ \ A_{d}=2\boldsymbol{e_{d}}.$$
    Then $\{A_k\}_{k=1}^{d}\subset \mathcal{T}(\mathbf{S})$, and for any $\lambda>0$, we have
    \[
    \frac{\lambda}{2}\left(A_d+\sum_{j=1}^\mathbf{k}A_{2j}\right)+\frac{\lambda}{6}\,\sum_{j=1}^\mathbf{k}A_{2j-1}=\left(\lambda,\cdots,\lambda\right)\in\R^d.
    \]
    Let $\lambda_0=\frac{6}{2d+1}$, then  $ (\mathbf{k}+1)\frac{\lambda_0}{2}+\mathbf{k}\frac{\lambda_0}{6}=1$ and $\boldsymbol{\lambda_0}:=(\lambda_0,\cdots,\lambda_0)$ is in the convex combination of $\{A_k\}_{k=1}^d$ with strictly positive coefficients. In view of the convexity of $\mathcal{N}(\mathbf{S})$ and the definition, we know $d_\mathbf{S}\le \lambda_0$. However, since $\mathcal{T}(\mathbf{S})\subset \Theta_{d-1}^{\boldsymbol{w_d}}$, 
    we know $\mathcal{N}(\mathbf{S})\subset \Theta_{d-1}^{\boldsymbol{w_d}}+\R_{+}^{d}$. Noting that 
    $$\Big\{(\varrho,\cdots,\varrho)\in\R^d:\varrho>0\Big\}\,\cap\, \Theta_{d-1}^{\boldsymbol{w_d}}=(\lambda_0,\cdots,\lambda_0)=:\boldsymbol{\lambda_0},$$ 
    we get $d_{\mathbf{S}}\ge \lambda_0$, and thus $d_{\mathbf{S}} = \lambda_0=\frac{6}{2d+1}$. 

    Now we prove that $k_{\mathbf{S}}=1$. Notice that $\mathcal{A}=\Theta_{d-1}^{\boldsymbol{w_d}}\cap \mathcal{N}(\mathbf{S})$ is a compact face of $\mathcal{N}(\mathbf{S})$, while $\{A_j\}_{j=1}^d \subset \mathcal{A}$ is linearly independent in $\R^d$, so it is also affinely independent. Thus  $\dim_{\R^d}(\mathcal{A})=d-1$, then by  \cite[Theorem 3.5, Exercises 3.1]{AB83}, we know 
     $\boldsymbol{\lambda_0}\in\mbox{ri}\,(\mathcal{A})$, which shows that $k_{\mathbf{S}}=1$.
     
     In conclusion, we finish the proof of Theorem \ref{thm-conj} by Theorem \ref{thm-new}.\\

\begin{remark}
  From the proof, for odd $d\ge 5$ we have reduced the problem on uniform estimates of DW to $J(t,\, Y+P, \,\psi)$, where $Y$ is as in \eqref{eq-R} and $P$ is a linear perturbation. Moreover, similar argument can also be applied to the setting of the DKG.\\
\end{remark}

\section{Space-time Estimates and nonlinear equation}\label{sec-nonli}

In this section, as a convention,  for a function $h=h(x,t)$ on $\Z^d\times\R$, we write $h=h(t)$ for simplicity. Moerover, for two positive functions $h_1$ and $h_2$, we write $h_1\lesssim h_2$ if there exists $C>0$ independent of $t\in\R$ such that $h_1\leq C h_2$, other dependence of the constants may be ignored. \\

\subsection{$l^p \rightarrow l^q$ estimates}\label{ssec-lplq} In this part we give the proof of Theorem \ref{thm-lplq}.

\begin{proof}[Proof of Theorem \ref{thm-lplq}]

By Theorem \ref{main-thm} and the Plancherel theorem, we have $$|G(t)|_{{\infty}} \lesssim (1+|t|)^{-(\frac{3}{2})^-}\quad\mbox{and}\quad\  |G(t)|_2 \lesssim 1.$$  Therefore, by interpolation we deduce that for any $ 2 \leqslant k \leqslant \infty$, we have
  \begin{equation}\label{equ-Gk}
        |G(t)|_k \lesssim (1+|t|)^{-\zeta_k} .
    \end{equation}

    By Young's inequality, for $r\geq 2$ satisfying
    $1+ \frac{1}{q} = \frac{1}{r}+\frac{1}{p}$, it holds that
  \begin{equation*}
        \begin{aligned}
            |u(t)|_q=|G(t)*f|_q   \leqslant |G(t)|_r |f|_p \lesssim  (1+|t|)^{-\zeta_r} |f|_p 
            =  (1+|t|)^{-\beta_{p,q}} |f|_p .
        \end{aligned}
    \end{equation*}

    For the second assertion of Theorem \ref{thm-lplq}, we use the Sobolev embedding (cf. e.g. \cite[Theorem 3.6]{HD15}), 
for any $g \in l^2$, we have
    \begin{equation}\label{equ-l24}
        \begin{aligned}
            |G(t)*g|_4  & \lesssim  |\nabla(G(t)*g)|_2 = \Big|\mathcal{F}\left(\nabla(G(t)*g)\right)\Big|_2 = \Big| \omega \cdot \mathcal{F}({G(t)*g}) \Big|_2 \\
            &= \Big| \omega \cdot \widehat{G}(t) \cdot \widehat{g}\Big|_2  = \Big| \omega \cdot \frac{\sin t\omega}{\omega} \cdot \widehat{g}\Big|_2  = \Big| \sin t\omega \cdot\widehat{g} \Big|_2  \leqslant |\widehat{g}|_2 = |g|_2 .
        \end{aligned}
    \end{equation}
Moreover, by (\ref{equ-Im}) it holds that 
$$G(x-y,t) = \overline{G(y-x,t)} =G(y-x,t).$$
Hence $\langle G(t)*f,g\rangle=\langle f,G(t)*g \rangle.$
   Combining this with (\ref{equ-l24}), we obtain
   \begin{equation*}
        \begin{aligned}
            |G(t)*f|_2 = \sup_{g \in l^2,|g|_{2}=1} {|\langle G(t)*f,g\rangle|}   \le \sup_{g \in l^2,|g|_{2}=1} {|f|_{\frac{4}{3}} \, |G(t)*g|_4} \lesssim |f|_{\frac{4}{3}} .
        \end{aligned}
    \end{equation*}
    
   By Young's inequality, we also have
    \begin{equation*}
        |G(t)*f|_{\infty} \le (1+|t|)^{-(\frac{3}{2})^-} |f|_1 .
    \end{equation*}
Thus for $2<q<\infty$, the Riesz-Thorin interpolation theorem yields 
    \begin{equation*}
        |u(t)|_q=|G(t)*f|_q \lesssim (1+|t|)^{-\zeta_q}|f|_p,
    \end{equation*}   
    which completes the proof.

\end{proof}

\begin{remark}
   In Theorem \ref{thm-lplq}, we only sharpen the estimate
   on a specific segment, one may conclude similar assertions on other segments using the same method. 
\end{remark}

\begin{corollary}\label{thm-stri/lplq}
     Let $u$ be the solution to \eqref{equ-orig} on $\Z^4$ with $g,F \equiv 0$, if
    \begin{equation*}
        1\leqslant p < r \leqslant \infty,\quad 0<q<\infty\quad\mbox{and}\quad \frac{1}{p}>\frac{1}{2} + \frac{1}{r}+\frac{1}{3q},
    \end{equation*}
   then there exists $C=C(p,q,r)$ such that
    \begin{equation*}
        \|u\|_{L^q_t l^r} \leqslant C |f|_{p}.
    \end{equation*}
   
    Furthermore, if $(\frac{1}{p},\frac{1}{r})$ lies on the segment with vertices $(\frac{3}{4},\frac{1}{2})$ and $(1,0)$, then the same conclusion holds for $(p,q,r)$ satisfying
    $\frac{1}{r}+\frac{1}{3q}< \frac{1}{2}$
    or $(p,q,r)=(\frac{4}{3},\infty,2)$.
\end{corollary}

\begin{proof}
   By Theorem \ref{thm-lplq}, for any $1\leqslant p < r \leqslant \infty$ and $q>\beta_{p,r}^{-1}$, we have
\begin{equation*}
    \begin{aligned}
    \|u\|_{L^q_t l^r}^q 
    & = \|G(t)*f\|_{L^q_t l^r}^q  \lesssim\,|f|_{p}^q \, \int_{\R}\ (1+|t|)^{-q\beta_{p,r} }\,dt \lesssim |f|_p^q.
    \end{aligned}
\end{equation*}

In addition, if $\frac{1}{r}=(-2)\frac{1}{p}+2$ and $1 \leqslant p \leqslant \frac{4}{3}$, the index $\beta_{p,r}$ can be replaced by $\zeta_r$. Then we finish the proof.

\end{proof}

Corollary \ref{thm-stri/lplq} is the first version of Strichartz estimates. In next section, we prove another version, i.e. Theorem \ref{thm-stri}, by a different method. \\

\subsection{Strichartz estimates}\label{ssec-stri/tao}

Before proving Theorem \ref{thm-stri}, we first give some preparations. The following theorem is a direct consequence of \cite[Theorem 1.2]{KT98}.

\begin{theorem}\label{thm-tao}
    Let $d\ge 2$ and a family of operators $\{U(t)\}_{t\in\R}$ satisfy 
    \begin{equation*}
        |U(t)h|_{2} \lesssim |h|_2, \quad\forall\, (h,t)\in l^2(\Z^d)\times\R,
    \end{equation*}
    and the truncated decay estimate for some $\sigma>0$,
    \begin{equation*}
        |U(t)(U(s))^*h|_{\infty} \lesssim (1+|t-s|)^{-\sigma}|h|_{1},\quad \forall\, (h,t,s)\in l^1(\Z^d)\times\R^2,
    \end{equation*}
    where $(U(t))^*$ denotes the adjoint of $U(t)$. If 
    $$q,r \geqslant 2,\quad(q,r,\sigma) \neq (2,\infty,1)\quad\mbox{and}\quad\frac{1}{q} \leqslant \sigma \left(\frac{1}{2}-\frac{1}{r}\right),$$
    while $(\Tilde{q},\Tilde{r})$ satisfies the same conditions, then 
    \begin{equation*}
        \|U(t)F_1\|_{L^{q}_t l^{r}} \lesssim |F_1|_2\quad\mbox{and}\quad 
        \left\|\int_{s<t}U(t)(U(s))^*F_2(s) \, ds\right\|_{L^{q}_t l^{r}} \lesssim \|F_2\|_{L^{\Tilde{q}'}_t l^{\Tilde{r}'}}
    \end{equation*}
    hold for all $(F_1,F_2)\in l^2(\Z^d)\times L_t^{\tilde{q}'} l^{\tilde{r}'}$.
\end{theorem}

The following Lemma \ref{lem-omega} can be found in \cite{MCRN17,MS22}, we give an alternative proof here.

\begin{lemma}\label{lem-omega}
Let $d\ge 2$, then 
    \begin{equation*}
    (1+|x|)^{-d-\nu}\lesssim \left| \int_{\mathbb{T}^d} e^{ix\cdot \xi} \,\omega(\xi)^{\nu}\, d \xi\right| \lesssim  (1+|x|)^{-d-\nu}
\end{equation*}
holds for all $x\in\Z^d$ with $|x|$ large enough, where $\nu= \pm 1$.
\end{lemma}
\begin{proof}
    We only prove the case $\nu=-1$, the other case is similar. For any $s\in\R$,  $$\mathbf{u}_{s}(\xi):=|\xi|^s,\quad\forall\,\xi\in\R^d.$$
    By the distribution theory, if $s\in (-d,0)$, there exists $C=C(d,s)$ such that
\begin{equation}\label{*}
\mathcal{F}(\mathbf{u}_s)=C\,\mathbf{u}_{-d-s}.
\end{equation}
Since $\omega^{-1}\in L^1({\mathbb{T}^d})$ for $d\geq 2$, we  have 
\begin{equation*}
    \int_{\mathbb{T}^d} e^{ix\cdot \xi} \omega(\xi)^{-1}\, d \xi=\int_{\mathbb{R}^d}e^{ix\cdot \xi} \omega(\xi)^{-1}\eta(\xi)\,d \xi=\mathcal{F}{\left(\omega^{-1}\eta\right)}(x), \ \ \forall\,x\in\Z^d,
\end{equation*}
where $\eta$ is as in (\ref{eq-per}). Now we split
$$
\mathcal{F}{\left(\omega^{-1}\eta\right)}= 
\mathcal{F}\left((\omega^{-1}\eta\, \mathbf{u}_1-1)\mathbf{u}_{-1}\right)+\mathcal{F}(\mathbf{u}_{-1})=:\mathcal{F}(g_1)+\mathcal{F}(\mathbf{u}_{-1}).
$$
Inserting a smooth cut-off $\psi$ with $\psi=1$ near the origin, we get
$$
\mathcal{F}{\left(\omega^{-1}\eta\right)}=\mathcal{F}(g_1\psi)+\mathcal{F}(g_1(1-\psi))+\mathcal{F}(\mathbf{u}_{-1})=:\mathbf{O}_1+\mathbf{O}_2+\mathbf{O}_3.
$$

By (\ref{*}) with $w=-1$, we get that $\mathbf{O}_3$ behaves like $\mathbf{u}_{-d+1}$. Moreover, for any $\beta\in\mathbb{N}^d$, a direct computation and an induction yield
$
\partial^{\beta}(g_1\psi)=\mathcal{O}(\mathbf{u}_{1-|\beta|}).
$ When $|\beta|=d$, we get $\partial^{\beta}(g_1\psi)\in L^{1}(\mathbb{R}^d)$. Then $\mathbf{O}_{1}=\mathcal{O}(\mathbf{u}_{-d})$.

As for $\mathbf{O}_2$, a similar argument shows that 
$
\partial^{N}(g_1(1-\psi))\in L^{1}(\mathbb{R}^d)
$
for all $N\in\mathbb{N}$ large enough. Therefore, we finish the proof.

\end{proof}

In the sequel, for proper functions $h_1$ and $h_2$, the opertator $h_1(D)$ is defined as
\begin{equation}\label{def-D}
    h_1(D)h_2 := \mathcal{F}^{-1}(h_1(\omega)) * h_2,\quad\mbox{with}\ \,  \omega \,\ \mbox{in \eqref{equ-gree}}.
\end{equation}
Note that the operator $D$ coincides with $\sqrt{-\Delta}$ in operator theory, see (\ref{equ-solu-semigroup}). 

\begin{lemma}\label{lem-d^-1}
      Let $d\ge 2$ and $b>\frac{d}{d-1}$, then $\frac{1}{D}: l^a(\Z^d) \rightarrow l^b(\Z^d)$ is bounded for any $1\le a\le \frac{bd}{b+d}$.
\end{lemma}

\begin{proof}
By Lemma \ref{lem-omega} with $\nu=-1$, we get
\begin{equation}\label{equ-5.6}
        \bigg|\frac{1}{D}\,f\bigg|_{b} \lesssim \big|(1+|\cdot|)^{-d+1} * |f|\big|_{b}= \sup_{ h \in l^{b'},|h|_{b'}=1}  \langle(1+|\cdot|)^{-d+1} * |f|,h\rangle.
    \end{equation}
     For any $b>\frac{d}{d-1}$, the discrete Hardy-Littlewood-Sobolev inequality (cf. e.g. \cite[Section 1]{HLY15}) and H\"{o}lder inequality yield
    \begin{equation*}
        \begin{aligned}
        \left\langle(1+|\cdot|)^{-d+1} * |f|,h\right \rangle  
            &  \lesssim |f|_{\frac{bd}{b+d}} + |f|_b  \lesssim |f|_{\frac{bd}{b+d}},
        \end{aligned}
    \end{equation*}
     where in the last ``$\lesssim$" we used the fact that $l^{\frac{bd}{b+d}}(\Z^d) \subset l^{b}(\Z^d)$ since $\frac{bd}{b+d} < b$.
     Inserting this into (\ref{equ-5.6}), we finish the proof.
     
\end{proof}

Now we give the proof of Theorem \ref{thm-stri}.

\begin{proof}[Proof of Theorem \ref{thm-stri}]
    By Duhamel's formula,
    \begin{equation}\label{equ-duha}
        u(x,t) = \cos (tD) g(x) + \frac{\sin (tD)}{D} f(x) + \int_0^t \frac{\sin(t-s)D}{D}F(x,s)ds.
    \end{equation}
    
Let the truncated operators $U_{\pm}(t):=\chi_{[0,\infty)}(t) \,e^{\pm itD}$ in Theorem \ref{thm-tao},  where $\chi_{[0,\infty)}(t)=1$ if $t\geq 0$, and $\chi_{[0,\infty)}(t)=0$ if $t<0$. One can check that the following assertions hold:
   \begin{gather*}
       \mbox{(a)}  \
        |U_{\pm}(t)f_1|_2 \leqslant |f_1|_2,\quad
        \mbox{(b)} \ (U_{+}(t))^*=U_{-}(t),\\
        \mbox{(c)} \ e^{itD}e^{isD}f_1 = e^{i(t+s)D}f_1, \quad \forall\, (t,s)\in\R^2.
   \end{gather*}
   
    By (b), (c) and Young's inequality, we deduce that
    \begin{equation*}
        \begin{aligned}
            |U_{+}(t)(U_{+}(s))^*f_2|_{\infty} &  \leqslant |e^{itD}e^{-isD}f_2|_{\infty}  = |\mathcal{F}(e^{i(t-s)\omega})*f_2|_{\infty}\\ &\leqslant |\mathcal{F}(e^{i(t-s)\omega})|_{\infty} |f_2|_1 
            \lesssim (1+|t-s|)^{-\frac{3}{2}}\log(2+|t-s|)|f_2|_1,
        \end{aligned}
    \end{equation*}
    which is the truncated decay estimate. {In the last ``$\lesssim$" we used the fact that 
    \begin{equation}\label{equ-fourier}
        |\mathcal{F}(e^{it\omega})|_{\infty}\lesssim (1+|t|)^{-\frac{3}{2}}\log(2+|t|),\quad\forall\,t\in\R.
    \end{equation} 
    The proof of (\ref{equ-fourier}) is  similar to that of Theorem \ref{main-thm}, see also the remark in \cite[Page 692]{S98}.}
    Now Theorem \ref{thm-tao} implies 
    \begin{equation}\label{equ-tao1}
    \|e^{itD} f\|_{L^q_t l^r} \lesssim |f|_{2}\quad\mbox{and}\quad 
    \left\|\int_0^t e^{i(t-s)D}F(x,s)ds \right\|_{L^q_t l^r} \lesssim \|F\|_{L_t^{\tilde{q}'} l^{\tilde{r}'}},
    \end{equation}
    for any $t,s \in \R$, where $(q,r)$ and $(\tilde{q},\tilde{r})$ satisfy the conditions in Theorem \ref{thm-tao} with $\sigma=(\frac{3}{2})^-$.
   
    By (\ref{equ-duha}), (\ref{equ-tao1}) and Euler's formula, we have
    \begin{equation}\label{equ-duhatao}
        \begin{aligned}
        \|u\|_{L^q_t l^r} 
        & \leqslant \| \cos (tD) g\|_{L^q_t l^r} +\left\|\frac{\sin (tD)}{D} f \right\|_{L^q_t l^r} +\; \left\|\int_0^t \frac{\sin(t-s)D}{D}F(x,s)ds \right\|_{L^q_t l^r} \\
        & \lesssim |g|_{2} + \left|\frac{1}{D}f\right|_{2} + \left\|\frac{1}{D}F\right\|_{L_t^{\tilde{q}'} l^{\tilde{r}'}}\lesssim |g|_2 + |f|_{{\frac{4}{3}}} + \|F\|_{L_t^{\tilde{q}'} l^{ 
 \frac{4\tilde{r}'}{4+\tilde{r}'}  }},
    \end{aligned}
    \end{equation}
 where in the last ``$\lesssim$" we used Lemma \ref{lem-d^-1}. Thus we finish the proof.

    
\end{proof}

As a direct corollary, we have:
\begin{corollary}\label{cor-stri/tao}
    Let $u$ be the solution to \eqref{equ-orig} on $\Z^4$ with $g,F \equiv 0$, then for any $q,r\ge 2$ satisfying $\frac{1}{q}<\frac{3}{2}(\frac{1}{2}-\frac{1}{r})$, there exists $C=C(q,r)>0$ such that 
    \begin{equation*}
        \|u\|_{L^q_t l^r} \leqslant C |f|_{{\frac{4}{3}}}.
    \end{equation*}
\end{corollary}

\begin{remark}\label{rmk-compare}
    Theorem \ref{thm-stri/lplq} 
    and Corollary \ref{cor-stri/tao} are both Strichartz type estimates. When $p=\frac{4}{3}$, 
     Corollary \ref{cor-stri/tao} admits a wider range of $(q,r)$. However, for a fixed pair $(p,r)$, Theorem \ref{thm-stri/lplq} may provide a wider range of $q$ than Corollary \ref{cor-stri/tao}. 
    Moreover,  
    Theorem \ref{thm-stri/lplq} admits initial data arbitrarily close to $l^2$, which may be useful in taking limit processes.
\end{remark}

By the same token, Strichartz estimates on $\Z^2$ and $\Z^3$ can be obtained by these methods except one case. In fact, there are some difficulties if we apply Theorem \ref{thm-tao} on $\Z^2$ directly. 
Now we state the results on $\Z^3$ analogue to Theorem \ref{thm-stri}.

\begin{theorem}\label{thm-stri-d=3}
    Let $u$ be the solution to \eqref{equ-orig} on $\Z^3$. If indices $q,r,\widetilde{q},\widetilde{r}$ satisfy
    \begin{equation*}
        q,r,\tilde{q},\tilde{r} \geq 2,\quad\frac{1}{q} \leq \frac{7}{6}\left(\frac{1}{2}-\frac{1}{r}\right)\quad\mbox{and}\quad
        \frac{1}{\tilde{q}} \leq \frac{7}{6}\left(\frac{1}{2}-\frac{1}{\tilde{r}}\right),    
    \end{equation*}
    then there exists $C=C(q,r,\widetilde{q},\widetilde{r})$ such that
    \begin{equation*}
        \|u\|_{L^q_t l^r} \leqslant C \left(|g|_{2} + |f|_{{\frac{6}{5}}} + ||F||_{L_t^{\tilde{q}'} l^{\frac{3\tilde{r}'}{3+\tilde{r}'} }  }\right).
    \end{equation*}
\end{theorem}

\subsection{Nonlinear equations}\label{ssec-non}

We use space-time estimates to obtain global well-posedness for small initial data. 

\begin{theorem}\label{thm-global-1}
    Let $k\ge 4$, $F(s) = |s|^{k-1}s$, $\forall\, s\in\R$ and $g=0$ in \eqref{equ-orig}.
  If $f\in l^1(\Z^4)$ with $|f|_1$ sufficiently small, then for any $2\le p\le \infty$, the global solution to \eqref{equ-orig} exists in $l^p$. We denote it by $u_k$, then it holds that
    \begin{equation}\label{equ-uk}
        |u_k(t)|_p \lesssim (1+|t|)^{-\zeta_p},\quad \forall\, t\in\R.
    \end{equation}
\end{theorem}

\begin{proof}
    We use the contraction mapping theorem. Let $\zeta_k$ be as in Theorem \ref{thm-lplq}. We consider the metric space
    \begin{equation*}
        \mathcal{M}:=\left\{h : \|h\|_{\mathcal{M}}= \sup_{t\in\R}(1+|t|)^{\zeta_k}|h(\cdot,t)|_{k} \leqslant 2C_0|f|_{1}\right\},
    \end{equation*}
    with $C_0=C_0(k,|f|_1)$ to be determined later. We also define $\Lambda$ on $\mathcal{M}$,
    \begin{equation*}
        \Lambda h :=\Lambda h(t) = G(t)*f + \int_{0}^{t}G(t-s)*F(h(s)) \, ds,\quad\forall\, h\in\mathcal{M}.
    \end{equation*}
    If $h\in\mathcal{M}$, we prove that $\Lambda h \in \mathcal{M}$. By \eqref{equ-Gk} and the fact that $|F(u(s))| \le |u(s)|_{k}^k$, $\forall\, s\in\R$, it holds that
    \begin{equation*}
            |\Lambda h(t)|_k 
            \leqslant (1+|t|)^{-\zeta_k}|f|_1 + \left|\int_{0}^{t} (1+|t-s|)^{-\zeta_k}|h(s)|^k_k\, ds\right|,\quad\forall\,t\in\R,
    \end{equation*}
    where we also used Theorem \ref{thm-lplq} with $\beta_{1,k}=\zeta_k$. Note that $1+|t| \leqslant (1+|t-s|)(1+|s|)$ for all $(s,t)\in\R^2$, we obtain
    \begin{equation*}
        (1+|t|)^{\zeta_k} |\Lambda h(t)|_k \leqslant |f|_1 + \|h\|_{\mathcal{M}}^k \int_{\R}\, (1+|s|)^{\zeta_k - k \zeta_k}\, ds=:|f|_1 + \|h\|_{\mathcal{M}}^k \,\mathbf{U}_k.
    \end{equation*}
    Since $k\ge 4$, then $\zeta_k-k \zeta_k < -1$ and $\mathbf{U}_k<\infty$. 
    Thus for $C_0=C_0(k)$ large enough,
    \begin{equation*}
        \|\Lambda h\|_{\mathcal{M}} \leqslant |f|_1 + \mathbf{U}_{k} \|h\|_{\mathcal{M}}^k \leqslant C_0 (|f|_1+\|h\|_{\mathcal{M}}^k).
    \end{equation*}
    If $|f|_1$ is sufficiently small such that $(2C_0)^k |f|_1^k\leqslant |f|_1$, we have
    \begin{equation*}
        \|\Lambda h\|_{\mathcal{M}} \leqslant C_0 (|f|_1+\|h\|_{\mathcal{M}}^k) \leqslant 2 C_0 |f|_1,
    \end{equation*}
    which shows that $\Lambda u \in \mathcal{M}$. 
    
    Moreover, one can show that $\Lambda$ is a contraction. Indeed, there exists $C=C(k)$ such that 
    \begin{equation*}
        \begin{aligned}
        \|u_1-u_2\|_{\mathcal{M}}
            & \leqslant C \mathbf{U}_k (\|u_1\|_{\mathcal{M}}^{k-1}+\|u_2\|_{\mathcal{M}}^{k-1})\|u_1-u_2\|_{\mathcal{M}} \\
            & \leqslant 2C\mathbf{U}_k(2C_0|f|_1)^{k-1}\|u_1-u_2\|_{\mathcal{M}}<\frac{1}{2}\|u_1-u_2\|_{\mathcal{M}},\quad \forall \,u_1,u_2\in\mathcal{M},
        \end{aligned}
    \end{equation*}
    as long as $|f|_1$ is sufficiently small. Therefore, $\Lambda$ admits a fixed point in $\mathcal{M}$, which is a solution to (\ref{equ-orig}). 
    
    By the same token, we can prove \eqref{equ-uk} for general $p \geq 2$, which finishes the proof of Theorem \ref{thm-global-1}.
    
\end{proof}

\begin{theorem}\label{thm-global-2}
    Let $F(s) = |s|^{k-1}s$, $\forall\,s\in\R$ with some $k \geqslant 3$, $p_k=\frac{2k+1}{2k}$ and $q_k=\frac{2k+1}{2}$. If $f\in l^{p_k}(\Z^4)$ with $|f|_{p_k}$ sufficiently small, then the solution to \eqref{equ-orig} exists. We denote it by $v_k$, then it holds that
    \begin{equation*}
        |v_k(t)|_{q_k} \lesssim (1+|t|)^{-\frac{3}{2}\left(1-\frac{2}{q_k}\right)},\quad\forall\,t\in\R.
    \end{equation*}
\end{theorem}

This result can be proved in the same manner as Theorem \ref{thm-global-1}.\\

\section*{Acknowledgement}

The authors wish to express gratitude to Prof. J.-C. Cuenin, I. A. Ikromov and D. M\"{u}ller for their helpful explanations. C.Bi is indebted to Prof. Hong-Quan Li for many useful suggestions. B.H. is supported by NSFC, No. 11831004, and by Shanghai Science and Technology Program [Project No. 22JC1400100].\\
	


\printbibliography

@book {HHor90,
    AUTHOR = {H\"{o}rmander, Lars},
     TITLE = {The analysis of linear partial differential operators. {I}},
    SERIES = {Springer Study Edition},
   EDITION = {Second},
      NOTE = {Distribution theory and Fourier analysis},
 PUBLISHER = {Springer-Verlag, Berlin},
      YEAR = {1990},
     PAGES = {xii+440},
      ISBN = {3-540-52343-X},
   MRCLASS = {35-02 (42B10 46Fxx)},
  MRNUMBER = {1065136},
       DOI = {10.1007/978-3-642-61497-2},
       URL = {https://doi.org/10.1007/978-3-642-61497-2},
}

@book {Hor90,
    AUTHOR = {H\"{o}rmander, Lars},
     TITLE = {An introduction to complex analysis in several variables},
    SERIES = {North-Holland Mathematical Library},
    VOLUME = {7},
   EDITION = {Third},
 PUBLISHER = {North-Holland Publishing Co., Amsterdam},
      YEAR = {1990},
     PAGES = {xii+254},
      ISBN = {0-444-88446-7},
   MRCLASS = {32-01 (35N15)},
  MRNUMBER = {1045639},
}

@article {BGX00,
    AUTHOR = {Bahouri, Hajer and G\'{e}rard, Patrick and Xu, Chao-Jiang},
     TITLE = {Espaces de {B}esov et estimations de {S}trichartz
              g\'{e}n\'{e}ralis\'{e}es sur le groupe de {H}eisenberg},
   JOURNAL = {J. Anal. Math.},
  FJOURNAL = {Journal d'Analyse Math\'{e}matique},
    VOLUME = {82},
      YEAR = {2000},
     PAGES = {93--118},
      ISSN = {0021-7670,1565-8538},
   MRCLASS = {58J45 (22E25)},
  MRNUMBER = {1799659},
MRREVIEWER = {Tohru\ Ozawa},
       DOI = {10.1007/BF02791223},
       URL = {https://doi.org/10.1007/BF02791223},
}

@article {BFG16,
    AUTHOR = {Bahouri, Hajer and Fermanian-Kammerer, Clotilde and Gallagher,
              Isabelle},
     TITLE = {Dispersive estimates for the {S}chr\"{o}dinger operator on
              step-2 stratified {L}ie groups},
   JOURNAL = {Anal. PDE},
  FJOURNAL = {Analysis \& PDE},
    VOLUME = {9},
      YEAR = {2016},
    NUMBER = {3},
     PAGES = {545--574},
      ISSN = {2157-5045,1948-206X},
   MRCLASS = {35R03 (35B40 35J10)},
  MRNUMBER = {3518529},
MRREVIEWER = {Wayne\ M.\ Eby},
       DOI = {10.2140/apde.2016.9.545},
       URL = {https://doi.org/10.2140/apde.2016.9.545},
}

@article {D05,
    AUTHOR = {Del Hierro, Martin},
     TITLE = {Dispersive and {S}trichartz estimates on {H}-type groups},
   JOURNAL = {Studia Math.},
  FJOURNAL = {Studia Mathematica},
    VOLUME = {169},
      YEAR = {2005},
    NUMBER = {1},
     PAGES = {1--20},
      ISSN = {0039-3223,1730-6337},
   MRCLASS = {35H20 (22E25 33C45)},
  MRNUMBER = {2139639},
       DOI = {10.4064/sm169-1-1},
       URL = {https://doi.org/10.4064/sm169-1-1},
}

@article {IM11,
    AUTHOR = {Ikromov, Isroil A. and M\"{u}ller, Detlef},
     TITLE = {Uniform estimates for the {F}ourier transform of surface
              carried measures in {$\Bbb R^3$} and an application to
              {F}ourier restriction},
   JOURNAL = {J. Fourier Anal. Appl.},
  FJOURNAL = {The Journal of Fourier Analysis and Applications},
    VOLUME = {17},
      YEAR = {2011},
    NUMBER = {6},
     PAGES = {1292--1332},
      ISSN = {1069-5869,1531-5851},
   MRCLASS = {42B20 (42B25 44A15)},
  MRNUMBER = {2854839},
MRREVIEWER = {Abdelhamid\ Boussejra},
       DOI = {10.1007/s00041-011-9191-4},
       URL = {https://doi.org/10.1007/s00041-011-9191-4},
}

@article {K84,
    AUTHOR = {Karpushkin, V. N.},
     TITLE = {A theorem on uniform estimates for oscillatory integrals with
              a phase depending on two variables},
   JOURNAL = {Trudy Sem. Petrovsk.},
  FJOURNAL = {Moskovski\u{\i} Universitet. Trudy Seminara imeni I. G.
              Petrovskogo},
    NUMBER = {10},
      YEAR = {1984},
     PAGES = {150--169, 238},
      ISSN = {0321-2971},
   MRCLASS = {58G15 (32B30 58C27)},
  MRNUMBER = {778884},
MRREVIEWER = {R\'{e}mi\ Vaillancourt},
}

@article {G14,
    AUTHOR = {Greenblatt, Michael},
     TITLE = {Stability of oscillatory integral asymptotics in two
              dimensions},
   JOURNAL = {J. Geom. Anal.},
  FJOURNAL = {Journal of Geometric Analysis},
    VOLUME = {24},
      YEAR = {2014},
    NUMBER = {1},
     PAGES = {417--444},
      ISSN = {1050-6926,1559-002X},
   MRCLASS = {42B25},
  MRNUMBER = {3145929},
MRREVIEWER = {Anna\ K.\ Savvopoulou},
       DOI = {10.1007/s12220-012-9341-1},
       URL = {https://doi.org/10.1007/s12220-012-9341-1},
}

@article {PS97,
    AUTHOR = {Phong, D. H. and Stein, E. M.},
     TITLE = {The {N}ewton polyhedron and oscillatory integral operators},
   JOURNAL = {Acta Math.},
  FJOURNAL = {Acta Mathematica},
    VOLUME = {179},
      YEAR = {1997},
    NUMBER = {1},
     PAGES = {105--152},
      ISSN = {0001-5962,1871-2509},
   MRCLASS = {42B20 (35S30 47G30 58G15)},
  MRNUMBER = {1484770},
MRREVIEWER = {Andreas\ Seeger},
       DOI = {10.1007/BF02392721},
       URL = {https://doi.org/10.1007/BF02392721},
}

@article {S98,
    AUTHOR = {Schultz, Pete},
     TITLE = {The wave equation on the lattice in two and three dimensions},
   JOURNAL = {Comm. Pure Appl. Math.},
  FJOURNAL = {Communications on Pure and Applied Mathematics},
    VOLUME = {51},
      YEAR = {1998},
    NUMBER = {6},
     PAGES = {663--695},
      ISSN = {0010-3640,1097-0312},
   MRCLASS = {39A12 (35L05)},
  MRNUMBER = {1611132},
MRREVIEWER = {Spyridon\ Kamvissis},
       DOI = {10.1002/(SICI)1097-0312(199806)51:6<663::AID-CPA4>3.0.CO;2-5},
       URL =
       {https://doi.org/10.1002/(SICI)1097-0312(199806)51:6<663::AID-CPA4>3.0.CO;2-5},
}

@article {K83,
    AUTHOR = {Karpushkin, V. N.},
     TITLE = {Uniform estimates of oscillating integrals with a parabolic or
              a hyperbolic phase},
   JOURNAL = {Trudy Sem. Petrovsk.},
  FJOURNAL = {Moskovski\u{\i} Universitet. Trudy Seminara imeni I. G.
              Petrovskogo},
    NUMBER = {9},
      YEAR = {1983},
     PAGES = {3--39},
      ISSN = {0321-2971},
   MRCLASS = {58G15 (32B30 58C27)},
  MRNUMBER = {731895},
MRREVIEWER = {J.\ S.\ Joel},
}

@book {M21,
    AUTHOR = {Montaldi, James},
     TITLE = {Singularities, bifurcations and catastrophes},
 PUBLISHER = {Cambridge University Press, Cambridge},
      YEAR = {2021},
     PAGES = {xvii+430},
      ISBN = {978-1-107-15164-2; 978-1-316-60621-6},
   MRCLASS = {58Kxx (37Gxx)},
  MRNUMBER = {4273545},
MRREVIEWER = {D.\ R. J. Chillingworth},
       DOI = {10.1017/9781316585085},
       URL = {https://doi.org/10.1017/9781316585085},
}

@book {AGV12,
    AUTHOR = {Arnold, V. I. and Gusein-Zade, S. M. and Varchenko, A. N.},
     TITLE = {Singularities of differentiable maps. {V}olume 1},
    SERIES = {Modern Birkh\"{a}user Classics},
      NOTE = {Classification of critical points, caustics and wave fronts,
              Translated from the Russian by Ian Porteous based on a
              previous translation by Mark Reynolds,
              Reprint of the 1985 edition},
 PUBLISHER = {Birkh\"{a}user/Springer, New York},
      YEAR = {2012},
     PAGES = {xii+382},
      ISBN = {978-0-8176-8339-9},
   MRCLASS = {58Kxx},
  MRNUMBER = {2896292},
}

@article {V76,
    AUTHOR = {Var\v{c}enko, A. N.},
     TITLE = {Newton polyhedra and estimates of oscillatory integrals},
   JOURNAL = {Funkcional. Anal. i Prilo\v{z}en.},
  FJOURNAL = {Akademija Nauk SSSR. Funkcional\cprime nyi Analiz i ego
              Prilo\v{z}enija},
    VOLUME = {10},
      YEAR = {1976},
    NUMBER = {3},
     PAGES = {13--38},
      ISSN = {0374-1990},
   MRCLASS = {14B05 (14D05)},
  MRNUMBER = {422257},
MRREVIEWER = {Helmut\ Hamm},
}

@article {G18,
    AUTHOR = {Gilula, Maxim},
     TITLE = {Some oscillatory integral estimates via real analysis},
   JOURNAL = {Math. Z.},
  FJOURNAL = {Mathematische Zeitschrift},
    VOLUME = {289},
      YEAR = {2018},
    NUMBER = {1-2},
     PAGES = {377--403},
      ISSN = {0025-5874,1432-1823},
   MRCLASS = {42B20},
  MRNUMBER = {3803795},
MRREVIEWER = {Ayhan\ Serbetci},
       DOI = {10.1007/s00209-017-1956-2},
       URL = {https://doi.org/10.1007/s00209-017-1956-2},
}

@article {GV95,
    AUTHOR = {Ginibre, J. and Velo, G.},
     TITLE = {Generalized {S}trichartz inequalities for the wave equation},
   JOURNAL = {J. Funct. Anal.},
  FJOURNAL = {Journal of Functional Analysis},
    VOLUME = {133},
      YEAR = {1995},
    NUMBER = {1},
     PAGES = {50--68},
      ISSN = {0022-1236,1096-0783},
   MRCLASS = {46E40 (35B45 35L05 46E30 46M35 47F05)},
  MRNUMBER = {1351643},
MRREVIEWER = {Sadakazu\ Aizawa},
       DOI = {10.1006/jfan.1995.1119},
       URL = {https://doi.org/10.1006/jfan.1995.1119},
}

@article {S77,
    AUTHOR = {Strichartz, Robert S.},
     TITLE = {Restrictions of {F}ourier transforms to quadratic surfaces and
              decay of solutions of wave equations},
   JOURNAL = {Duke Math. J.},
  FJOURNAL = {Duke Mathematical Journal},
    VOLUME = {44},
      YEAR = {1977},
    NUMBER = {3},
     PAGES = {705--714},
      ISSN = {0012-7094,1547-7398},
   MRCLASS = {46F10 (35B40 81.35)},
  MRNUMBER = {512086},
MRREVIEWER = {R.\ Glassey},
       URL = {http://projecteuclid.org/euclid.dmj/1077312392},
}

@article {KT98,
    AUTHOR = {Keel, Markus and Tao, Terence},
     TITLE = {Endpoint {S}trichartz estimates},
   JOURNAL = {Amer. J. Math.},
  FJOURNAL = {American Journal of Mathematics},
    VOLUME = {120},
      YEAR = {1998},
    NUMBER = {5},
     PAGES = {955--980},
      ISSN = {0002-9327,1080-6377},
   MRCLASS = {35B45 (35L70 35Q55)},
  MRNUMBER = {1646048},
MRREVIEWER = {John\ Albert},
       URL =
              {http://muse.jhu.edu/journals/american_journal_of_mathematics/v120/120.5keel.pdf},
}

@article {SK05,
    AUTHOR = {Stefanov, Atanas and Kevrekidis, Panayotis G.},
     TITLE = {Asymptotic behaviour of small solutions for the discrete
              nonlinear {S}chr\"{o}dinger and {K}lein-{G}ordon equations},
   JOURNAL = {Nonlinearity},
  FJOURNAL = {Nonlinearity},
    VOLUME = {18},
      YEAR = {2005},
    NUMBER = {4},
     PAGES = {1841--1857},
      ISSN = {0951-7715,1361-6544},
   MRCLASS = {37K60 (34B45 35Q53 35Q55)},
  MRNUMBER = {2150357},
MRREVIEWER = {Gulmaro\ Corona-Corona},
       DOI = {10.1088/0951-7715/18/4/022},
       URL = {https://doi.org/10.1088/0951-7715/18/4/022},
}

@article {BG17,
    AUTHOR = {Borovyk, Vita and Goldberg, Michael},
     TITLE = {The {K}lein-{G}ordon equation on {$\Bbb{Z}^2$} and the quantum
              harmonic lattice},
   JOURNAL = {J. Math. Pures Appl. (9)},
  FJOURNAL = {Journal de Math\'{e}matiques Pures et Appliqu\'{e}es.
              Neuvi\`eme S\'{e}rie},
    VOLUME = {107},
      YEAR = {2017},
    NUMBER = {6},
     PAGES = {667--696},
      ISSN = {0021-7824,1776-3371},
   MRCLASS = {35R02 (35L05 42B20 81Q05 81Q35)},
  MRNUMBER = {3650321},
MRREVIEWER = {Baris\ Evren\ Ugurcan},
       DOI = {10.1016/j.matpur.2016.10.002},
       URL = {https://doi.org/10.1016/j.matpur.2016.10.002},
}

@article {CI21,
    AUTHOR = {Cuenin, Jean-Claude and Ikromov, Isroil A.},
     TITLE = {Sharp time decay estimates for the discrete {K}lein-{G}ordon
              equation},
   JOURNAL = {Nonlinearity},
  FJOURNAL = {Nonlinearity},
    VOLUME = {34},
      YEAR = {2021},
    NUMBER = {11},
     PAGES = {7938--7962},
      ISSN = {0951-7715,1361-6544},
   MRCLASS = {37K60 (35L05 35R02 42B20 81Q05)},
  MRNUMBER = {4331254},
MRREVIEWER = {Yifei\ Wu},
       DOI = {10.1088/1361-6544/ac2b86},
       URL = {https://doi.org/10.1088/1361-6544/ac2b86},
}

@article {D74,
    AUTHOR = {Duistermaat, J. J.},
     TITLE = {Oscillatory integrals, {L}agrange immersions and unfolding of
              singularities},
   JOURNAL = {Comm. Pure Appl. Math.},
  FJOURNAL = {Communications on Pure and Applied Mathematics},
    VOLUME = {27},
      YEAR = {1974},
     PAGES = {207--281},
      ISSN = {0010-3640,1097-0312},
   MRCLASS = {58G99 (35S99 58C25)},
  MRNUMBER = {405513},
MRREVIEWER = {Alan\ Weinstein},
       DOI = {10.1002/cpa.3160270205},
       URL = {https://doi.org/10.1002/cpa.3160270205},
}

@book{D93,
series = {Cambridge topics in mineral physics and chemistry ; 4},
issn = {9780521392938},
isbn = {9780511619885},
year = {1993},
title = {Introduction to lattice dynamics},
language = {eng},
author = {Dove, Martin T.},
}

@book {S96,
    AUTHOR = {Schultz, Pete},
     TITLE = {Nonlinear wave equations in multidimensional lattices},
      NOTE = {Thesis (Ph.D.)--New York University},
 PUBLISHER = {ProQuest LLC, Ann Arbor, MI},
      YEAR = {1996},
     PAGES = {160},
      ISBN = {978-0591-08365-1},
   MRCLASS = {99-05},
  MRNUMBER = {2694758},
       URL =
              {http://gateway.proquest.com/openurl?url_ver=Z39.88-2004&rft_val_fmt=info:ofi/fmt:kev:mtx:dissertation&res_dat=xri:pqdiss&rft_dat=xri:pqdiss:9702055},
}

@misc{IS21,
      title={Uniform estimates for oscillatory integrals with smooth phase}, 
      author={Ikromov, Isroil A. and Safarov Akbar Raxmanovich},
      year={2021},
      eprint={2111.03747},
      archivePrefix={arXiv},
      primaryClass={math.CA}
}

@article {RSK21,
    AUTHOR = {Ruzhansky, Michael and Safarov, Akbar R. and Khasanov,
              Gafurjan A.},
     TITLE = {Uniform estimates for oscillatory integrals with homogeneous
              polynomial phases of degree 4},
   JOURNAL = {Anal. Math. Phys.},
  FJOURNAL = {Analysis and Mathematical Physics},
    VOLUME = {12},
      YEAR = {2022},
    NUMBER = {6},
     PAGES = {Paper No. 130, 15},
      ISSN = {1664-2368,1664-235X},
   MRCLASS = {35B50 (26D10 42B20)},
  MRNUMBER = {4493247},
MRREVIEWER = {Berikbol\ T.\ Torebek},
       DOI = {10.1007/s13324-022-00747-w},
       URL = {https://doi.org/10.1007/s13324-022-00747-w},
}

@article {K94,
    AUTHOR = {Karpushkin, V. N.},
     TITLE = {The leading term of the asymptotics of oscillatory integrals
              with a phase of the series {$T$}},
   JOURNAL = {Mat. Zametki},
  FJOURNAL = {Matematicheskie Zametki},
    VOLUME = {56},
      YEAR = {1994},
    NUMBER = {6},
     PAGES = {131--133},
      ISSN = {0025-567X,2305-2880},
   MRCLASS = {41A60 (58G18)},
  MRNUMBER = {1330607},
       DOI = {10.1007/BF02266700},
       URL = {https://doi.org/10.1007/BF02266700},
}

@book{Z72,
publisher = {Cambridge University Press},
isbn = {9780521297332},
year = {1972},
title = {Principles of the Theory of Solids},
copyright = {Cambridge University Press 1972},
language = {eng},
author = {Ziman, J. M.},
}

@book {FH10,
    AUTHOR = {Feynman, Richard P. and Hibbs, Albert R.},
     TITLE = {Quantum mechanics and path integrals},
   EDITION = {Emended},
      NOTE = {Emended and with a preface by Daniel F. Styer},
 PUBLISHER = {Dover Publications, Inc., Mineola, NY},
      YEAR = {2010},
     PAGES = {xii+371},
      ISBN = {978-0-486-47722-0; 0-486-47722-3},
   MRCLASS = {81S40},
  MRNUMBER = {2797644},
}

@book{W11,
series = {Pseudo-Differential Operators},
volume = {5},
publisher = {Springer Nature},
isbn = {3034801165},
year = {2011},
title = {Discrete Fourier Analysis},
copyright = {Springer Basel AG 2011},
language = {eng},
address = {Basel},
author = {Wong, M. W},
}

@book {S93,
    AUTHOR = {Stein, Elias M.},
     TITLE = {Harmonic analysis: real-variable methods, orthogonality, and
              oscillatory integrals},
    SERIES = {Princeton Mathematical Series},
    VOLUME = {43},
      NOTE = {With the assistance of Timothy S. Murphy,
              Monographs in Harmonic Analysis, III},
 PUBLISHER = {Princeton University Press, Princeton, NJ},
      YEAR = {1993},
     PAGES = {xiv+695},
      ISBN = {0-691-03216-5},
   MRCLASS = {42-02 (35Sxx 43-02 47G30)},
  MRNUMBER = {1232192},
MRREVIEWER = {Michael\ Cowling},
}

@book {AB83,
    AUTHOR = {Br\o ndsted, Arne},
     TITLE = {An introduction to convex polytopes},
    SERIES = {Graduate Texts in Mathematics},
    VOLUME = {90},
 PUBLISHER = {Springer-Verlag, New York-Berlin},
      YEAR = {1983},
     PAGES = {viii+160},
      ISBN = {0-387-90722-X},
   MRCLASS = {52A25 (05B30 52-01)},
  MRNUMBER = {683612},
MRREVIEWER = {D.\ Barnette},
}

@book {BJ77,
    AUTHOR = {Bazaraa, Mokhtar S. and Jarvis, John J.},
     TITLE = {Linear programming and network flows},
 PUBLISHER = {John Wiley \& Sons, New York-London-Sydney},
      YEAR = {1977},
     PAGES = {x+565},
   MRCLASS = {90CXX (90-01)},
  MRNUMBER = {443985},
MRREVIEWER = {A.\ C.\ Williams},
}

@article {HLY15,
    AUTHOR = {Huang, Genggeng and Li, Congming and Yin, Ximing},
     TITLE = {Existence of the maximizing pair for the discrete
              {H}ardy-{L}ittlewood-{S}obolev inequality},
   JOURNAL = {Discrete Contin. Dyn. Syst.},
  FJOURNAL = {Discrete and Continuous Dynamical Systems. Series A},
    VOLUME = {35},
      YEAR = {2015},
    NUMBER = {3},
     PAGES = {935--942},
      ISSN = {1078-0947,1553-5231},
   MRCLASS = {35A23 (42B25)},
  MRNUMBER = {3277179},
MRREVIEWER = {Nguyen\ Lam},
       DOI = {10.3934/dcds.2015.35.935},
       URL = {https://doi.org/10.3934/dcds.2015.35.935},
}

@article {HD15,
    AUTHOR = {Hua, Bobo and Mugnolo, Delio},
     TITLE = {Time regularity and long-time behavior of parabolic
              {$p$}-{L}aplace equations on infinite graphs},
   JOURNAL = {J. Differential Equations},
  FJOURNAL = {Journal of Differential Equations},
    VOLUME = {259},
      YEAR = {2015},
    NUMBER = {11},
     PAGES = {6162--6190},
      ISSN = {0022-0396,1090-2732},
   MRCLASS = {39A12 (05C50 05C63 35B40 35B65 35J92 35R02 47H20)},
  MRNUMBER = {3397320},
       DOI = {10.1016/j.jde.2015.07.018},
       URL = {https://doi.org/10.1016/j.jde.2015.07.018},
}

@book {V80,
    AUTHOR = {Vinogradov, I. M.},
     TITLE = {The method of trigonometric sums in the theory of numbers (in Russian)},
   EDITION = {Second},
 PUBLISHER = {``Nauka'', Moscow},
      YEAR = {1980},
     PAGES = {144},
   MRCLASS = {10G05 (10G10)},
  MRNUMBER = {603100},
}

@article {MS22,
    AUTHOR = {Michta, Emmanuel and Slade, Gordon},
     TITLE = {Asymptotic behaviour of the lattice {G}reen function},
   JOURNAL = {ALEA Lat. Am. J. Probab. Math. Stat.},
  FJOURNAL = {ALEA. Latin American Journal of Probability and Mathematical
              Statistics},
    VOLUME = {19},
      YEAR = {2022},
    NUMBER = {1},
     PAGES = {957--981},
      ISSN = {1980-0436},
   MRCLASS = {60K35 (33C10 41A60 82B27)},
  MRNUMBER = {4448687},
       DOI = {10.30757/alea.v19-38},
       URL = {https://doi.org/10.30757/alea.v19-38},
}

@article {MCRN17,
    AUTHOR = {Michelitsch, T. M. and Collet, B. A. and Riascos, A. P. and
              Nowakowski, A. F. and Nicolleau, F. C. G. A.},
     TITLE = {Recurrence of random walks with long-range steps generated by
              fractional {L}aplacian matrices on regular networks and simple
              cubic lattices},
   JOURNAL = {J. Phys. A},
  FJOURNAL = {Journal of Physics. A. Mathematical and Theoretical},
    VOLUME = {50},
      YEAR = {2017},
    NUMBER = {50},
     PAGES = {505004, 29},
      ISSN = {1751-8113,1751-8121},
   MRCLASS = {60G50 (60G22 60G51 60J10 81Q35 94C15)},
  MRNUMBER = {3738798},
       DOI = {10.1088/1751-8121/aa9008},
       URL = {https://doi.org/10.1088/1751-8121/aa9008},
}

@article {FT04,
    AUTHOR = {Friedman, Joel and Tillich, Jean-Pierre},
     TITLE = {Wave equations for graphs and the edge-based {L}aplacian},
   JOURNAL = {Pacific J. Math.},
  FJOURNAL = {Pacific Journal of Mathematics},
    VOLUME = {216},
      YEAR = {2004},
    NUMBER = {2},
     PAGES = {229--266},
      ISSN = {0030-8730,1945-5844},
   MRCLASS = {05C50 (35L05)},
  MRNUMBER = {2094545},
MRREVIEWER = {Slobodan\ Simi\'{c}},
       DOI = {10.2140/pjm.2004.216.229},
       URL = {https://doi.org/10.2140/pjm.2004.216.229},
}

@ARTICLE{HH20,
       author = {Han, Fengwen and Hua, Bobo},
        title = {Uniqueness class of the wave equation on graphs},
         year = 2020,
          doi = {10.48550/arXiv.2009.12793},
archivePrefix = {arXiv},
       eprint = {2009.12793},
 primaryClass = {math.AP},
       adsurl = {https://ui.adsabs.harvard.edu/abs/2020arXiv200912793H},
}

@article {LX22,
    AUTHOR = {Lin, Yong and Xie, Yuanyuan},
     TITLE = {Application of {R}othe's method to a nonlinear wave equation on graphs},
   JOURNAL = {Bull. Korean Math. Soc.},
  FJOURNAL = {Bulletin of the Korean Mathematical Society},
    VOLUME = {59},
      YEAR = {2022},
    NUMBER = {3},
     PAGES = {745--756},
      ISSN = {1015-8634,2234-3016},
   MRCLASS = {35L05 (35R02 58J45)},
  MRNUMBER = {4432439},
       DOI = {10.4134/BKMS.b210445},
       URL = {https://doi.org/10.4134/BKMS.b210445},
}

@ARTICLE{LX19,
       author = {{Lin}, Yong and {Xie}, Yuanyuan},
        title = {The existence of the solution of the wave equation on graphs},
        year = {2019},
        doi = {10.48550/arXiv.1908.02137},
        archivePrefix = {arXiv},
        eprint = {1908.02137},
        primaryClass = {math.AP},
        adsurl = {https://ui.adsabs.harvard.edu/abs/2019arXiv190802137L},
}

@article {VDC23,
    AUTHOR = {van der Corput, J. G.},
     TITLE = {Neue zahlentheoretische {A}bsch\"{a}tzungen},
   JOURNAL = {Math. Ann.},
  FJOURNAL = {Mathematische Annalen},
    VOLUME = {89},
      YEAR = {1923},
    NUMBER = {3-4},
     PAGES = {215--254},
      ISSN = {0025-5831,1432-1807},
   MRCLASS = {99-04},
  MRNUMBER = {1512147},
       DOI = {10.1007/BF01455979},
       URL = {https://doi.org/10.1007/BF01455979},
}

@book {B17,
    AUTHOR = {Barlow, Martin T.},
     TITLE = {Random walks and heat kernels on graphs},
    SERIES = {London Mathematical Society Lecture Note Series},
    VOLUME = {438},
 PUBLISHER = {Cambridge University Press, Cambridge},
      YEAR = {2017},
     PAGES = {xi+226},
      ISBN = {978-1-107-67442-4},
   MRCLASS = {60-02 (05C63 05C81 60J10 60J45)},
  MRNUMBER = {3616731},
MRREVIEWER = {Nicolas\ Curien},
       DOI = {10.1017/9781107415690},
       URL = {https://doi.org/10.1017/9781107415690},
}

@book {Gri18,
    AUTHOR = {Grigoryan, Alexander},
     TITLE = {Introduction to analysis on graphs},
    SERIES = {University Lecture Series},
    VOLUME = {71},
 PUBLISHER = {American Mathematical Society, Providence, RI},
      YEAR = {2018},
     PAGES = {viii+150},
      ISBN = {978-1-4704-4397-9},
   MRCLASS = {60J10 (05C25 05C50 05C63 05C81 58J35)},
  MRNUMBER = {3822363},
MRREVIEWER = {Serguei\ Popov},
       DOI = {10.1090/ulect/071},
       URL = {https://doi.org/10.1090/ulect/071},
}

\end{document}